\documentclass[12pt,reqno]{amsart}
\usepackage{amssymb,amsmath,amsthm,enumerate}

\DeclareMathOperator{\sign}{sign}
\DeclareMathOperator{\Ran}{Ran}

\DeclareMathOperator{\Ker}{Ker}

\DeclareMathOperator{\spec}{spec}

\DeclareMathOperator{\supp}{supp}
\DeclareMathOperator{\dist}{dist}
\DeclareMathOperator*{\slim}{s-lim}
\DeclareMathOperator{\diag}{diag}

\renewcommand\Im{\hbox{{\rm Im}}\,}
\renewcommand\Re{\hbox{{\rm Re}}\,}
\newcommand{\abs}[1]{\lvert#1\rvert}

\newcommand{\norm}[1]{\lVert#1\rVert}

\newcommand{\bbR}{{\mathbb R}}
\newcommand{\bbC}{{\mathbb C}}

\newcommand{\bbS}{{\mathbb S}}
\newcommand{\bbT}{{\mathbb T}}
\newcommand{\bbY}{{\mathbb Y}}
\newcommand{\bbP}{{\mathbf P}}
\newcommand{\bM}{{\mathbf M}}
\newcommand{\wh}{\widehat}
\newcommand{\calH}{{\mathcal H}}
\newcommand{\calK}{{\mathcal K}}
\newcommand{\calF}{\mathcal{F}}

\newcommand{\calN}{\mathcal{N}}

\newcommand{\calZ}{\mathcal{Z}}
\newcommand{\calY}{\mathcal{Y}}

\newcommand{\calB}{\mathcal{B}}
\newcommand{\calU}{\mathcal{U}}
\newcommand{\frakh}{\mathfrak{h}}
\newcommand{\DD}{\mathcal{D}}

\numberwithin{equation}{section}

\theoremstyle{plain}
\newtheorem{theorem}{\bf Theorem}[section]
\newtheorem{lemma}[theorem]{\bf Lemma}
\newtheorem{proposition}[theorem]{\bf Proposition}
\newtheorem{assumption}[theorem]{\bf Assumption}
\newtheorem{corollary}[theorem]{\bf Corollary}

\theoremstyle{definition}
\newtheorem{definition}[theorem]{\bf Definition}

\theoremstyle{remark}
\newtheorem*{remark*}{\bf Remark}
\newtheorem{remark}[theorem]{\bf Remark}
\newtheorem{example}[theorem]{\bf Example}

\newcommand{\wt}{\widetilde}
\newcommand{\eps}{\varepsilon}

\newcommand{\X}{q}
\newcommand{\loc}{\text{loc}}
\newcommand{\ac}{\text{\rm ac}}

\begin{document}

\title[Piecewise continuous functions of operators]{Spectral theory of piecewise 
continuous functions of self-adjoint operators}

\author{Alexander Pushnitski}
\address{Department of Mathematics,
King's College London, 
Strand, London, WC2R~2LS, U.K.}
\email{alexander.pushnitski@kcl.ac.uk}

\author{Dmitri Yafaev}
\address{Department of Mathematics, University of Rennes-1, 
Campus Beaulieu, 35042, Rennes, France}
\email{yafaev@univ-rennes1.fr}



\begin{abstract}
Let $H_0$, $H$ be a pair of self-adjoint operators for which
the standard assumptions of the smooth version of scattering 
theory hold true. We give an explicit description of the absolutely 
continuous spectrum of the operator $\DD_\theta=\theta(H)-\theta(H_0)$ for piecewise
continuous functions $\theta$. This description involves the 
scattering matrix for the pair $H_0$, $H$, evaluated at the 
discontinuities of $\theta$. 
We also prove that the singular continuous spectrum of 
$\DD_\theta$ is empty and that  the   
eigenvalues of this operator have finite multiplicities and may accumulate 
only to the ``thresholds"  of  the absolutely 
continuous spectrum of $\DD_\theta$. 
Our approach relies on the construction of ``model" operators 
for each jump of the function $\theta$. 
These model operators are defined as certain symmetrised Hankel operators 
which admit explicit spectral analysis. 
We develop the multichannel scattering theory for the set of model operators 
and the operator $\theta(H)-\theta(H_0)$. 
As a by-product of our approach, we also construct
the  scattering theory for general symmetrised Hankel operators with  piecewise
continuous symbols.
\end{abstract}

\maketitle

\section{Introduction}\label{sec.a}

\subsection{Overview}\label{a1}

Let $H_0$ and $H$ be self-adjoint operators    and suppose that the difference 
$V=H-H_0$ is a compact operator.  
If $\theta$ is a continuous function which tends to zero at infinity then the difference
\begin{equation}
\DD_\theta=\theta(H)-\theta(H_0)
\label{a0}
\end{equation}
is also compact. On the contrary, if $\theta$ has discontinuities, then the operator
 $\DD_\theta$ may acquire  the (absolutely) continuous   spectrum. This phenomenon was observed in \cite{KM} in a concrete example and established in \cite{Push} under fairly general assumptions.

Our goal here is to study the structure of the operator $\DD_\theta$ for piecewise continuous functions under assumptions on 
$H_0$, $H$ typical for smooth scattering theory (see, e.g., \cite{Kuroda,Yafaev}). Roughly speaking, these assumptions mean that the perturbation $V=H-H_{0}$ is an integral operator with a sufficiently smooth kernel in the spectral representation of the ``unperturbed" operator $ H_{0}$. Under our
assumptions  the scattering matrix $S(\lambda)$ for the pair $H_0$,  $H$ is well defined  for $\lambda$ in the absolutely continuous (a.c.) 
spectrum of the operator $H_{0}$. The scattering matrix is a unitary operator in some auxiliary Hilbert space $\calN$ which is the fiber space in the spectral representation of $H_{0}$. We denote by $N=\dim\calN $ the multiplicity of the spectrum of the operator $H_{0}$ in a neighbourhood of the point $\lambda$ (the number $N =N(\lambda) $ can be finite or infinite).    
Moreover, the operator $S(\lambda)-I$ is  compact,  
so that the spectrum of $S(\lambda)$ consists of  eigenvalues $\{\sigma_n(\lambda)\}_{n=1}^N$  lying
on the unit circle in $\bbC$.
Eigenvalues of $S(\lambda)$ distinct from $1$ have finite multiplicities and  can accumulate only to the point $1$.

  We suppose that $\theta(\lambda)$ is a continuous function except at points $\lambda_{1},\ldots, \lambda_{L}$, $L<\infty$,  where it has {\it jump discontinuities}. That is, at each of these 
  points $\lambda_\ell$, the limits $\theta(\lambda_\ell \pm 0)$ exist and are finite, but
  $\theta(\lambda_\ell+0) \neq \theta(\lambda_\ell-0)$. We denote the jumps of $\theta$ by
  \begin{equation}
\varkappa_\ell
=
\theta(\lambda_\ell+0)-\theta(\lambda_\ell-0)\neq 0.
\label{b8b}
\end{equation} 
We prove that the  a.c.\   
spectrum of the operator $\DD_\theta$   consists of the union of the intervals:
\begin{equation}
\spec_\ac \DD_\theta =  \bigcup_{\ell=1}^L \bigcup_{n=1}^{N_{\ell}} [-a_{n\ell},a_{n\ell}],
\quad
a_{n\ell}=\tfrac12\abs{\varkappa_\ell}\abs{\sigma_n(\lambda_\ell)-1}, 
\quad 
N_\ell=N(\lambda_\ell). 
\label{eq:a0}
\end{equation}
Here and in similar formulas below describing the a.c.\  spectrum, we use two conventions:
\begin{enumerate}[(i)]
\item
the union is taken over all non-trivial intervals, i.e.\  if $a_{n\ell}=0$, then the corresponding
interval is dropped from the union; 
\item
each interval contributes multiplicity one to the a.c.\  spectrum of $\DD_\theta$. 
That is, denoting by $\mathsf{A}(\Lambda)$ the operator of multiplication 
by $\lambda$ in $L^2(\Lambda,d\lambda)$, one can state \eqref{eq:a0} more precisely 
as follows: 
the  a.c.\  part of $\DD_\theta$ is unitarily equivalent 
to the orthogonal sum
$$
\DD_\theta^{(\ac)} \simeq\bigoplus_{\ell=1}^L  \bigoplus_{n=1}^{N_{\ell}}  \mathsf{A}([-a_{n\ell}, a_{n\ell}]).
$$
\end{enumerate}

 We also prove that the singular continuous spectrum 
of $\DD_\theta$ is empty, the eigenvalues of $\DD_\theta$ can accumulate
only to $0$ and to the points $\pm a_{n\ell}$, and 
all eigenvalues of $\DD_\theta$ distinct from $0$ and
$\pm a_{n\ell}$ have finite multiplicity.

It follows from \eqref{eq:a0} that $\spec_\ac \DD_\theta=\varnothing$ if and only   if $S(\lambda_{\ell})=I$ for all  $ \ell=1,\ldots, L$. In the latter case the operator $\DD_\theta$ is compact. We emphasize that only the jumps $\varkappa_\ell$ of $\theta(\lambda)$ 
at the points $\lambda_{\ell}$  of   discontinuity and the spectrum of  the scattering matrix $S(\lambda_{\ell})$ at these points are essential for our construction.

All of our results apply to the case where $H_{0}$ and $H$ are the free and perturbed Schr\"odinger operators -- see Example~\ref{Schr}.

We study the operator $\DD_\theta$ in the spectral representation  of   $ H_{0}$.    It turns out that the structure of  $\DD_\theta$ is naturally described in terms of the operators defined in the space 
$L^2(\bbR;  \calN )$ by the formula
\begin{equation}
M_\Xi=  P_- \Xi P_+ + P_+ \Xi^*P_- .
\label{b9a}
\end{equation}
Here $\calN$ is an auxiliary space of the dimension $N=N_{1}+\cdots +N_{L}$, $P_{\pm}$ are the orthogonal projections onto the Hardy classes $H^2_{\pm} (\bbR;  \calN )\subset L^2(\bbR;\calN)$ 
and   $\Xi$ is the operator of multiplication by the symbol $\Xi (\lambda)$ which, for each $\lambda$, is a bounded operator in the space $\calN$. For obvious reasons, we call operators \eqref{b9a} \emph{symmetrised Hankel operators}, SHOs for short. It is important that SHOs are automatically self-adjoint and, for a particular choice of $\Xi$,  admit an explicit diagonalization.   

We show that the  operators $\DD_\theta$ and $M_\Xi$ are, in some sense, close to each other if  the symbol   is constructed     by the formula
\begin{equation}
\Xi(\lambda)=  (2\pi i)^{-1} (S (\lambda)-I) \theta(\lambda).
\label{eq:MS}
\end{equation}
Thus the SHO $M_\Xi$ with such symbol plays the role of a model operator for $\DD_\theta$.    We  analyse the operator $M_\Xi$ and, as a consequence, establish the spectral results for $\DD_\theta$ mentioned above.
  
We emphasize that our assumptions on $H_{0}$ are local, i.e.\  its diagonalization is required only in neighbourhoods of the discontinuity points $\lambda_{1},\ldots, \lambda_{L}$.

\subsection{Main ideas of the approach}\label{sec.a1t}

Roughly speaking, our approach relies on the construction of scattering theory for the pair $M_\Xi, \DD_{\theta}$, that is, on the comparison of asymptotic behaviour as $t\to\pm\infty$ of 
  functions $\exp(-iM_\Xi t)f_{0}$ and $\exp(-i\DD_{\theta} t)f$  for elements $f_{0}$ and $f$ in the a.c.\  subspaces of the operator $M_\Xi$ and $\DD_{\theta}$, respectively. It turns out that   the functions $\exp(-iM_\Xi t)f_{0}$
    are asymptotically concentrated as $t\to\pm\infty$ in neighbourhoods of the points
$\lambda_{1},\ldots, \lambda_{L}$. This means that every discontinuity of $\theta$ yields its own band of the a.c.\  spectrum. 

To handle this situation, we introduce the model operators
$M_{\Xi_{\ell}} $  for all discontinuity points $\lambda_{\ell}$, $\ell =1, \ldots,  L$. The model operators
$M_{\Xi_{\ell}} $ are   defined as SHO in the space $L^2(\bbR;  \calN )$ with the symbols $\Xi_{\ell}$,
each of which containins only one jump. 
We choose these symbols in such a way that, up to smooth terms, 
the sum $\Xi_{1}+\cdots+ \Xi_{L} $ equals the function $\Xi$ defined by \eqref{eq:MS}.
It is important that each operator $M_{\Xi_{\ell}} $  can be explicitly diagonalized. Then we develop the scattering theory for the set of model operators $M_{\Xi_{1}}, \ldots,  M_{\Xi_L}$ and the operator $\DD_{\theta}$. To be more precise, we prove the existence of wave operators for all pairs $M_{\Xi_{\ell}},  \DD_{\theta}$, $\ell=1,\ldots, L$. The ranges of these wave operators for different $\ell$ are orthogonal to each other, and their orthogonal sum exhausts the a.c.\  subspace of the operator $\DD_{\theta}$. The results of this type are known as the asymptotic completeness of wave operators. Our proofs of these results   require a version of multichannel scattering theory  constructed in our earlier publication \cite{PY2}.  In the important particular case $L=1$ the multichannel scheme is not necessary, and it suffices to apply the usual results of smooth   scattering theory to the pair $M_{\Xi } $, $\DD_{\theta}$.

The construction of the model operators $M_{\Xi_{\ell}} $ relies on the existence of some bounded Hankel operator with simple a.c.\  spectrum which can be explicitly diagonalized.
The choice of such Hankel operator is not unique. We proceed from the Mehler operator, that is, the Hankel operator in the space $L^2(\bbR_{+} )$ with the integral kernel $\pi^{-1} (t+s)^{-1}$. Alternatively, we could have used the Hankel operator   with integral kernel $\pi^{-1} (t+s)^{-1} e^{-t-s}$  diagonalized by W.~Magnus and M.~Rosenblum.

We use the smooth method of scattering theory and work in the spectral representation of the operator $H_{0}$. Thus we ``transplant" the operator $\DD_{\theta}$ by an isometric (but not necessary unitary) transformation into the space  $L^2(\bbR;  \calN )$.  An important and technically difficult step is to localize the problem onto  a neighbourhood of the set $\{\lambda_{1},\ldots, \lambda_{L}\}$. It turns out that, after such a localization, the transplanted operator $\DD_{\theta}$ is close to the SHO $M_{\Xi}$ with symbol \eqref{eq:MS}. 

As a by-product of our approach, we develop the scattering theory for general SHOs $M_{\Xi } $ with piecewise continuous symbols $\Xi$.
 This   theory   is one of the key ingredients of our analysis 
and is perhaps of  interest  in its own sake. 

 We note that the SHOs introduced here are very well adapted to the study of discontinuous
 functions of self-adjoint operators and their spectral theory is simpler than that of standard Hankel operators. We plan to apply   the approach of the present paper to the usual   Hankel operators with discontinuous symbols 
 in a separate publication.

\subsection{History}\label{a1x}

The link between operators $\DD_\theta $ (for smooth functions $\theta$) and Hankel operators with the symbol $\theta$ was discovered in \cite{Peller1} and has been applied by V.~V.~Peller and his collaborators to the estimates of various norms of $\DD_\theta $.

The analysis of $\DD_\theta $ for discontinuous $\theta$ was initiated in \cite{KM, Push}.
Formula \eqref{eq:a0} first appeared in \cite{Push} under relatively stringent assumptions on the perturbation $V$ and for $\theta$ being the characteristic function of a half-line. In \cite{Push} a mixture of trace class and smooth methods of   scattering theory has been used.

 To a certain extent, this paper can be considered as a continuation of \cite{PY1} where   the purely smooth approach  has been applied to the study of the operator $\DD_\theta^2$. The main difference between \cite{PY1} and this work is   that here we analyse    the operator  $\DD_\theta$ directly whereas in \cite{PY1} only  the  spectral properties of the operator $\DD_\theta^2$ were considered. Thus the approach of \cite{PY1} does not capture information about the structure of the operator 
$\DD_\theta$ and, in particular, about its eigenfunctions. Another important difference is that here we treat  arbitrary piecewise continuous functions $\theta$ with finite limits at $\pm\infty$ whereas in \cite{PY1} 
only the case of $\theta$ being the characteristic function of a half-line 
was considered.  

Under somewhat less restrictive assumptions than here, it was shown in \cite{Push2}
that   the essential spectrum of $\DD_\theta$  coincides with the union of the intervals in the r.h.s.\  of \eqref{eq:a0}. This  result is of course   consistent with formula  \eqref{eq:a0} for the a.c.\  spectrum of $\DD_\theta$.  The operators 
$\DD_{\theta_{\varepsilon}}$ for smooth functions $\theta_{\varepsilon}$ with supports shrinking as $\varepsilon \to 0$ to some point $\lambda_{0}$   were studied in \cite{Push1}.

As far as the spectral theory of Hankel operators with piecewise continuous symbols is concerned, we first note 
S.~Power's characterisation \cite{Power} of the essential spectrum. In the self-adoint case, the absolutely continuous spectrum was described      in \cite{Howland1} by
J.~Howland  who used the trace class method. Moreover, he applied  in \cite{Howland2}    the Mourre method  to perturbations  of the Carleman operator and  proved   the absence of singular continuous spectrum in this case. 
 To a certain extent, the paper \cite{Howland1} can be considered as a precursor of our results on SHOs.

\subsection{The structure of the paper}\label{a1y}

The basic objects of   scattering theory are introduced in   Section~\ref{sec.b}, 
where we also  state the precise assumptions on the operators $H_{0}$ and $H$ 
specific for the smooth  scattering theory. 
In particular, in Subsection~\ref{sec.c} we summarize  
the results of \cite{PY2} concerning the  multichannel version of the scattering theory. 
This theory is used  in the study of both SHO
 \eqref{b9a}  and  the operators $\DD_{\theta}$.

In Sections~\ref{sec.e} and \ref{sand} we collect  diverse analytic results which are used in Section~\ref{sec.ee} for the study of SHOs and  in Section~\ref{sec.b5} for the study of the operators $\DD_{\theta}$.  In Section~\ref{sec.e} we diagonalize explicitly some special SHO that will be used as a model operator.  In Section~\ref{sand} we prove the compactness of Hankel operators sandwiched by some singular weights.

Spectral and scattering theory of   SHOs with piecewise continuous symbols is developed in Section~\ref{sec.ee}. Here the main results are    Theorems~\ref{thm.f8} and \ref{thm.f10}. 
 
 In Section~\ref{sec.b6}, we obtain  convenient representations for operator \eqref{a0}  sandwiched by appropriate functions of 
the operator $H_{0}$. These representations play an important role in our analysis and are perhaps of an independent interest.

     Our main results (Theorems~\ref{th.b1} and \ref{th.c3})  concerning the operators $\DD_{\theta}$ are stated and proven in Section~\ref{sec.b5}.


\subsection{Notation}\label{sec.a1a}
Let $\calH$ be a Hilbert space.  
We denote by $\calB=\calB(\calH)$ (resp.  by
$\mathfrak S_\infty=\mathfrak S_\infty(\calH)$) the class of all bounded (resp. compact) operators on $\calH$.
For a self-adjoint operator $A$, we denote by $E_A(\cdot)$   the projection-valued spectral measure of $A$; $\spec A $ is the spectrum of $A$  and $\spec_p A $ is its  point spectrum. We denote by $\calH^{(\ac )}_{A}$   the a.c.\  subspace of $A$, $P^{(\ac )}_{A}$ is the orthogonal projection onto $\calH^{(\ac)}_{A}$, $E^{(\ac)}_A(\cdot)= E_A(\cdot)P^{(\ac)}_{A} $ and  $\spec_\ac A $ is the a.c.\  spectrum of $A$.
For $K\in {\mathfrak S}_\infty(\calH)$ we denote by $s_n(K)$, $1\leq n\leq\dim\calH$,
the sequence of singular values  of $A$ (which may include zeros)
enumerated in the decreasing order with multiplicities taken into account. We denote by $C_0(\bbR; \calH)$ the 
space of all continuous functions $f:\bbR\to \calH$ such that 
$\norm{f(x)}_\calH\to0$ as $\abs{x}\to0$. 
We denote by $\chi_\Lambda$ the characteristic function of a set 
$\Lambda\subset \bbR$ and write 
$\bbR_\pm=\{x\in\bbR: \pm x>0\}$. 
We often use the same notation for a bounded function and 
the operator of multiplication by this function in the Hilbert space $L^2(\bbR)$. 
The same convention applies to operator valued functions.

  

\section{Spectral and scattering theory. Generalities}\label{sec.b}

\subsection{The strong smoothness}\label{sec.b2}

Let $A$ be a self-adjoint operator in a Hilbert space $\calH$. Suppose that $\delta$
is an open   interval where  the spectrum of $A$    is purely absolutely continuous with a constant multiplicity $N \leq\infty$. 
More explicitly, we assume that for some auxiliary Hilbert space $\calN $, 
 there exists a 
unitary operator $\calF $  
from $\Ran E_{A}(\delta )$ onto 
$L^2(\delta ;\calN )$, $\dim\calN=N$,   such that $\calF $
diagonalizes $A$: if $f\in\Ran E_{A}(\delta )$ then 
\begin{equation}
( \calF A f)(\mu)
=
\mu (\calF  f)(\mu), \quad \mu\in\delta .
\label{b2}\end{equation}


Let us formulate an assumption typical for smooth scattering theory. Let $Q$ be a bounded operator in   $\calH$.
Suppose that the operators $ Z (\mu):\calH\to\calN $  defined by the relation
\begin{equation}
 Z (\mu)f= (\calF    Q^* f)(\mu),\quad  \mu\in\delta ,
\label{eq:Z1}
\end{equation}
are compact and satisfy the estimates
\begin{equation}
\norm{Z (\mu)}\leq C,
\quad
\norm{Z (\mu)-Z (\mu')}\leq C\abs{\mu-\mu'}^\gamma ,
\quad
\mu, \mu'\in\delta , \quad \gamma\in (0, 1],
\label{b4}
\end{equation}
where the constant $C$ is independent of $\mu$, $\mu'$
in compact subintervals of $\delta $ and $\gamma\in (0, 1]$. Thus we accept the following

\begin{definition}\label{stsm}
We say that $Q$ {\it is strongly $A$-smooth on} 
$\delta $ with the  exponent $\gamma $ 
if,  for some diagonalization  $\calF $ of the operator $E_{A}(\delta ) A$ and 
for the operator $Z(\lambda)$ defined by \eqref{eq:Z1},   condition  \eqref{b4} is satisfied.  
\end{definition}
 
 It follows from \eqref{b2}, \eqref{eq:Z1} that for any bounded function $\varphi$ which is
compactly supported on $\delta $ and for all $f\in\calH$, we have
\begin{equation}
Q\varphi(A)f
=
\int_{-\infty}^\infty
Z ^* (\mu)  ( \calF   f)(\mu) \varphi(\mu) d\mu .
\label{b5}
\end{equation}

We emphasize that      in applications  the map $\calF $ emerges naturally. 

\subsection{Operators $H_{0}$ and $H$}\label{sec.b1}
Let $H_0$ and $V$ be  self-adjoint   operators in a Hilbert space $\calH$. For simplicity, we assume that the ``perturbation" $V$ is a bounded operator  so that the sum $H=H_0+V$ is self-adjoint on the domain of the operator $H_{0}$. Similarly to \cite{PY1}, all our constructions can easily be extended to a class of unbounded operators $V$, but we will not dwell upon this here.

Suppose that $V$ is factorized as $V=G^*V_0G$ 
where  
\begin{equation}
 V_0=V_0^* \in \calB(\calH),
\quad
G(\abs{H_0}+I)^{-1/2}\in \mathfrak S_\infty (\calH).
\label{a3}
\end{equation}
It is also convenient to assume that $\Ker G=\{0\}$.  Of course the resolvents $R_0(z)=(H_0-zI)^{-1}$ and $R(z)=(H-zI)^{-1}$ of the operators $H_{0}$ and $H$ are related by the usual identity
$$
R(z)-R_0(z)=-(GR_0(\overline z))^*{V_0}GR(z).
$$

Let   
\begin{equation}
T(z)=GR_0(z)G^* 
\label{a6a}
\end{equation}
be the sandwiched resolvent of the operator $H_{0}$.
It follows from the assumption \eqref{a3} that $ T(z)\in \mathfrak S_\infty$,   the operator $I+T(z){V_0}$ has a bounded inverse for all $z\in\bbC\setminus\bbR$ and  
\begin{equation}
R(z)=R_0(z)-(G R_0(\overline z))^* {V_0} (I+T(z){V_0})^{-1} GR_0(z).
\label{a7}
\end{equation}
 Using the notation   \eqref{a6a} we set
\begin{equation}
Y(z)=-2\pi i {V_0}(I+T(z){V_0})^{-1}, 
\quad \Im z\neq 0.
\label{b7}
\end{equation}
Then the resolvent identity 
\eqref{a7} can be rewritten in the  more concise form that we use below:
\begin{equation}
R(z)-R_0(z)
=
\frac{1}{2\pi i} (GR_0(\overline z))^* Y(z)GR_0(z), 
\quad \Im z\neq 0.
\label{d2}
\end{equation}
Note that
\begin{equation}
Y(\overline{z})=-Y^* (z)  .
\label{d4a}
\end{equation}

Let $ \Delta_\ell$,  $\ell=1,\ldots,L$,   be pairwise disjoint open   intervals and
$$\Delta=\Delta_{1}\cup\cdots \cup \Delta_{L}.$$
 We suppose that  the spectrum of $H_0$ in $\Delta_{\ell}$, $\ell=1,\ldots,L$,  is purely a.c.\  with a constant multiplicity $N_{\ell}\leq\infty$. 
More explicitly, we assume that for some auxiliary Hilbert space $\calN_{\ell}$, 
$\dim\calN_{\ell}=N_{\ell}$, there exists a 
unitary operator 
\begin{equation}
\calF_{\ell}: \Ran E_{H_0}(\Delta_{\ell})\
\to L^2(\Delta_{\ell};\calN_{\ell})
\label{eq:Zf}
\end{equation} 
 which diagonalizes $E_{H_0}(\Delta_{\ell}) H_0$. The corresponding operator \eqref{eq:Z1} will be denoted by $Z_{\ell}(\lambda)$:
\begin{equation}
 Z_{\ell} (\lambda)f= (\calF_{\ell}    G^* f)(\lambda),\quad  \lambda\in\Delta_{\ell} .
\label{eq:Z}\end{equation}

Let us summarize our assumptions:

\begin{assumption}\label{as1}

\begin{enumerate}[{\rm(A)}]
\item
$H_0$ has a purely a.c.\  spectrum with multiplicity $N_{\ell}$ on each interval $\Delta_{\ell}$, $\ell=1,\ldots, L$.
\item
$V$ admits a factorization $V=G^*V_0G$ satisfying \eqref{a3}.
\item
$G$ is strongly $H_0$-smooth on all intervals $ \Delta_{\ell}$.
\end{enumerate}
\end{assumption}


\begin{example}\label{Schr}
Let     $H_0=-{\pmb \Delta}$
be the  Laplace operator in the space $L^2(\bbR^d)$, $d\geq1$. 
Applying the Fourier transform $\Phi$, we see that the operator $\Phi H_{0}\Phi^*$ acts as the multiplication by $|\xi|^2$ in $L^2(\bbR^d; d\xi)$ (in the momentum representation). To diagonalize $H_{0}$, it remains to pass to the spherical coordinates in $\bbR^d$ and to make the change of variables $\lambda =|\xi|^2$. Now  $L=1$,   $\Delta_{1} =\bbR_{+}$, $\calN =L^2(\bbS^{d-1})$ (here $L^2(\bbS^{d-1})=\bbC^2$ if $d=1$) and the operator 
  $\calF_1 : L^2(\bbR^d)\to L^2(\bbR_{+}; \calN)$ diagonalizing $H_{0}$  is defined by the formula
$$
(\calF_1 f)(\lambda; \omega)
= 
2^{-1/2} \lambda^{(d-2)/4} (\Phi f)( \lambda^{1/2}\omega),\quad \lambda>0,\quad \omega\in \bbS^{d-1}.
$$
Of course, the operator $H_0$ has the purely a.c.\  spectrum $[0,\infty)$. It has  infinite multiplicity for $d\geq 2$ and multiplicity $2$ for $d=1$.

Suppose that $V$ acts as the multiplication by    a real  short-range   function $V(x)$, that is,  
\begin{equation}
\abs{V(x)}\leq C(1+\abs{x})^{-r}, 
\qquad 
r>1.
\label{eq:Schr}\end{equation}
Then Assumption~\ref{as1}(B) is satisfied with the operator $G$ acting as the multiplication by the function $(1+|x|)^{-r/2}$ and the operator $V_{0}$ acting as the multiplication by the bounded function $(1+|x|)^{r}V(x)$.  For this operator $G$,  Assumption~\ref{as1}(C)   is satisfied with an arbitrary $\gamma<(r-1)/2$ according to the Sobolev trace theorem.
\end{example}

It would have been too restrictive to assume that all discontinuities of $\theta$ lie in the same interval where the spectrum of $H_{0}$ has a constant multiplicity. This would exclude applications at least to two classes of operators $H_{0}$: (a)~operators with   spectral gaps (such as periodic operators, or Dirac operators) (b)~multichannel systems such as wave guides.

\subsection{The limiting absorption principle and spectral results}\label{sec.b3}
The following well-known result  (see, e.g., \cite[Theorems~4.7.2 and 4.7.3]{Yafaev}) is called the
 limiting absorption principle.

\begin{proposition}\label{pr1}
Let Assumption~$\ref{as1}$ hold true. 
Then  
the operator-valued function $T(z)$  defined by \eqref{a6a}
is   H\"older continuous for $\Re z\in \Delta$, $ \Im z \geq 0$.
In particular, the limits $T(\lambda+i0)$ exist in the operator norm and 
are H\"older continuous in $\lambda\in\Delta$.
Let $\mathfrak{N}  \subset \Delta$ be the set of $\lambda$ such that
the equation 
$$
f+T(\lambda+ i0){V_0}f=0
$$ 
has a non-zero solution $f\in\calH$. 
Then $\mathfrak{N}$
is closed and  
has the Lebesgue measure zero.  For all $\lambda\in\Omega:= \Delta\setminus \mathfrak{N}$, the inverse operator $(I+T(\lambda+i0)V_0)^{-1}$ 
exists, is bounded and is   H\"older continuous in
$\lambda\in\Omega$.
\end{proposition}

\begin{corollary}\label{pr1x}
Let the operator $Y(z)$ be defined by equation \eqref{b7}.  
The limits $Y(\lambda+i0)$ exist in the operator norm and 
are H\"older continuous in $\lambda\in\Omega$.
\end{corollary}

The limiting absorption principle can be supplemented by the following spectral results  (see, e.g., \cite[Theorems~4.7.9 and 4.7.10]{Yafaev}).

\begin{proposition}\label{pr2}
Let Assumption~$\ref{as1}$ hold true.
Then   the spectrum of $H$ in $\Omega$ 
is purely absolutely continuous.
If, in addition,   $\gamma >1/2$ in 
\eqref{b4}, then 
the singular continuous spectrum of $H$ in $\Delta$ is absent,
  $\mathfrak{N}=(\spec_p H ) \cap \Delta$, 
and the eigenvalues of $H$ in $\Delta$ have finite multiplicities
and can accumulate only to the endpoints of the intervals $\Delta_{\ell}$. In this case 
\begin{equation}
\Omega=\Delta\setminus \spec_p H.
\label{eq:Ome}\end{equation}
\end{proposition}

Note that, for  the   Schrodinger operator $H$ defined in Example~\ref{Schr}, the exponent $\gamma$ in \eqref{b4} can be an arbitrary small number if $r$ in \eqref{eq:Schr}  is close to $1$. Nevertheless, all assertions of Proposition~\ref{pr2} remain true in this case. Moreover, according to the Kato theorem the operator $H$ does not have positive eigenvalues so that, under assumption \eqref{eq:Schr}, $H$
 has purely absolutely continuous spectrum on $\bbR_{+}$  (see, e.g., \cite{RS}).


\subsection{Wave operators}\label{sec.b4}

For an interval (or the union of pairwise disjoint intervals)
$\Delta\subset{\mathbb R}$, the local wave operators are introduced by the relation
\begin{equation}
W_{\pm} (H,H_{0}; \Delta)
=
\slim_{t\to \pm \infty}e^{iHt}e^{-iH_{0}t}E_{H_0}^{(\ac)} (\Delta),
\label{eq:WO}\end{equation}
provided these strong limits exist. The wave operators are isometric on $\Ran E_{H_0}^{(\ac)}(\Delta)$, enjoy the intertwining property
$$
H W_{\pm} (H,H_{0}; \Delta)=W_{\pm} (H,H_{0}; \Delta) H_{0}
$$
 and
$$
\Ran W_{\pm} (H,H_{0}; \Delta)
\subset
\Ran E_H^{(\ac)} (\Delta).
$$
The wave operators are called complete   if this inclusion reduces to the equality.  
We need the following well-known result (see, e.g., \cite[Theorem~4.6.4]{Yafaev}).

\begin{proposition}\label{pr2wo}
Under Assumption~$\ref{as1}$  the local wave operators
$W_\pm(H,H_0; \Delta)$ exist and are complete.  In particular, the operators $H_{0}E_{H_{ 0}} (\Delta)$ and $H E_H^{({\rm ac})} (\Delta)$ are unitarily equivalent.
\end{proposition}

Note that if $E_{H_0}^{(\ac)} (\bbR\setminus\Delta)=0$, then  the operator $E_{H_0}  (\Delta)$ in definition \eqref{eq:WO} can be dropped and
the  global wave operators 
 $$
W_{\pm} (H,H_{0} )
=
\slim_{t\to \pm \infty}e^{iHt}e^{-iH_{0}t}P_{H_0}^{(\ac)} $$
 exist. This is the case  for the pair $H_{0}$, $H$ considered in Example~\ref{Schr}.

\subsection{Scattering matrix}\label{sec.b4a}
If the  wave operators
$W_\pm(H,H_0; \Delta)$ exist, then the (local) scattering operator is defined as
$$
\mathbf{S}=\mathbf{S}(H,H_0; \Delta)= W_+^* (H,H_0; \Delta)  W_-(H,H_0;\Delta).
$$
Moreover, the scattering operator $\mathbf{S}$
is unitary on the subspace $\Ran E_{H_0} (\Delta )$ if these wave operators are complete. The scattering operator $\mathbf{S}$
commutes with $H_{0}$, and therefore, for almost all $\lambda\in \Delta_{\ell}$, $\ell=1,\ldots, L$,
we have a representation
$$
( \calF_\ell \mathbf{S}  \calF_\ell^* f)(\lambda) =S(\lambda) f(\lambda) 
$$
where the operator  $S(\lambda):\calN_{\ell}\to\calN_{\ell}$  is  called the scattering matrix
for the pair of operators $H_0$, $H$. 
The scattering matrix $S(\lambda)$ where $\lambda\in \Delta_{\ell}$  is a unitary operator in $\calN_{\ell}$ if $\mathbf{S}$
is unitary.  Moreover, we have the following result (see \cite[Theorem~7.4.3]{Yafaev}).

\begin{proposition}\label{SM}
Let Assumption~$\ref{as1}$ hold. Let the operator $Y(z)$ be defined by formula \eqref{b7}, and let the operators $Z_{\ell}(\lambda)$ for $\lambda \in\Delta_{\ell}$ be  defined by formula \eqref{eq:Z}. Then the scattering matrix admits the  representation
\begin{equation}
S(\lambda)=I+Z_{\ell}(\lambda)Y(\lambda+i0)Z_{\ell}^*(\lambda) ,
\quad 
\lambda\in\Omega\cap \Delta_{\ell}.
\label{b8}
\end{equation}
\end{proposition}
 
The   representation \eqref{b8} implies that $S(\lambda)$ is 
a H\"older continuous function of $\lambda\in\Omega$. 
Since the operator $Y(\lambda+i0)$  
is bounded and $Z _{\ell}(\lambda)$ is compact, it follows that the 
operator $S(\lambda)-I$ is  compact.  
Thus, the spectrum of $S(\lambda)$ consists of  eigenvalues 
on the unit circle in $\bbC$
accumulating possibly only to the point 1. 
All eigenvalues $\sigma_n(\lambda)$, $1\leq n\leq N_\ell$  of $S(\lambda)$ distinct from $1$ have finite multiplicities; they are enumerated with multiplicities taken into account.


\subsection{Multichannel scheme}\label{sec.c}
In this subsection, we  
recall a key result from  \cite{PY2} which will   allow us to put together 
the contributions from each of the jumps of $\theta$. 
Let
 $A_{\ell}$, $\ell=1,2,\dots,L$, and $A_{\infty}$ be bounded self-adjoint operators. 
The abstract results below allow  us to construct spectral theory of the operator
\begin{equation}
A=  A_{1} +\cdots + A_{L} +A_{\infty}
\label{c3}
\end{equation}
and a smooth version of scattering theory for the pair of operators 
$$
A_1\oplus A_2\oplus\dots\oplus A_L \quad\text{ and }\quad A
$$
 under   certain smoothness assumptions on all pair products $A_j A_k$, $j\not=k$, and the operator $A_\infty$.  

 Spectral results are formulated in the following assertion.

\begin{proposition}[\cite{PY2}]  \label{pr.c1}
Let $A_{\ell}$, $\ell=1,\dots,L$, and $A_{\infty}$ be bounded self-adjoint operators 
in a Hilbert space $\calH$. 
Let $\delta\subset\bbR$ be an 
open interval such that $0\not\in\delta$.
Assume that the spectra of $A_1,\dots,A_L$ are purely a.c.\  on $\delta$ 
with constant multiplicities. 
Let $X$ be a bounded self-adjoint operator in $\calH$ such that 
$\Ker X=\{0\}$.
Assume   that:
\begin{itemize}
\item
For all $\ell=1,\dots,L$, the operator $X$ is strongly $A_\ell$-smooth 
on $\delta$ with an exponent $\gamma>1/2$.
\item
For all $1\leq j,k\leq L$, $j \not=k$, 
the operators $A_jA_{k}$ can be represented as
\begin{equation}
A_j A_{k} =X  K_{j  k} X  \quad
\text{ where}  \quad  K_{j k}\in \mathfrak{S}_{\infty}.
\label{c1}
\end{equation}
\item
The operator $A_\infty$ can be represented as
\begin{equation}
A_\inftyÊ=X K_\infty X  \quad
\text{ where}  \quad  K_\infty\in \mathfrak{S}_{\infty}.
\label{c2}
\end{equation}
\item
The operators $XA_{\ell}X^{-1} $ 
are bounded for all $\ell=1,\ldots,L$. 
\end{itemize}
Let the operator $A$ be defined by equality \eqref{c3}.
Then:
\begin{enumerate}[\rm (i)]
\item
The a.c.\  spectrum of the operator 
$A$ on $\delta$ has a uniform multiplicity   which is equal to the sum of
the multiplicities of the spectra of $A_1,\dots,A_L$ on $\delta$.
\item
The singular continuous spectrum of $A$ on $\delta$ is empty.
\item
The eigenvalues of $A$ on $\delta$ have finite multiplicites
and cannot accumulate at interior points of $\delta$.
\item
The operator-valued function $X(A-zI)^{-1}X$ is   continuous
in $z$ for $\pm\Im z\geq 0$, $\Re z\in\delta$, except at   the eigenvalues of $A$. 
\end{enumerate} 
\end{proposition}


The scattering theory for the set of operators $A_1,\ldots, A_{L}$ and the  operator $A$ is described in the following assertion. 
\begin{proposition}[\cite{PY2}] \label{pr.c2}
Let the hypotheses of Proposition~$\ref{pr.c1}$ hold true. 
Then:
\begin{enumerate}[\rm (i)]
\item
The local wave operators
$$
W_\pm(A,A_\ell;\delta)
=
\slim_{t\to\pm\infty} e^{iAt}e^{-iA_\ell t}E_{A_\ell} (\delta),
\quad 
\ell=1,\dots,L,
$$
exist and enjoy the intertwining property
$$
AW_\pm(A,A_\ell;\delta)=W_\pm(A,A_\ell;\delta)A_\ell,
\quad \ell=1,\dots, L.
$$
The wave operators are isometric and their ranges are orthogonal to each 
other, i.e.
$$
\Ran W_\pm(A,A_j;\delta)\perp \Ran W_\pm(A,A_k;\delta), 
\quad 1\leq j,k\leq L, \; j \not=k.
$$
\item
The asymptotic completeness holds: 
\begin{equation}
\Ran W_\pm(A,A_1;\delta)\oplus\dots\oplus \Ran W_\pm(A,A_L;\delta)
=
\Ran E^{({\rm ac})}_A(\delta).
\label{b12}
\end{equation}
\end{enumerate}
\end{proposition}

We note that the first statement of Proposition~\ref{pr.c1} is a direct consequence of Proposition~\ref{pr.c2}.

We also observe that, for $L=1$, Propositions~\ref{pr.c1} and \ref{pr.c2} reduce to the results of Subsections~\ref{sec.b3} and \ref{sec.b4} where $H_{0}$, $H$, $G$ and $\Delta$ play the roles of $A_{1}$, $A$, $X$ and $\delta$, respectively. In this case the assumptions $0\not\in\delta$ and $X A_{1}X^{-1}\in{\calB}$ (the last hypothesis of Proposition~\ref{pr.c1}) are not necessary.

\section{Symmetrised Hankel operators. A model operator}\label{sec.e}


Here we introduce the class of symmetrised Hankel operators (SHO) and diagonalize explicitly some special SHO. This SHO will be used as a model operator for the construction in Section~\ref{sec.ee} of the spectral theory of SHO with  piecewise continuous symbols as well as in Section~\ref{sec.b5} for the proof of our main results concerning the spectral theory of   operators \eqref{a0}.
 In view of the applications to the operator $\DD_{\theta}$, we consider SHOs with operator valued symbols. The main result of this section is Theorem~\ref{lma.e4}.
 
\subsection{Symmetrised Hankel operators}\label{sec.b5b}

Recall that
the Hardy space  $H^2_\pm(\bbR)\subset L^2(\bbR)$ 
is defined as the class of all functions $f \in L^2(\bbR)$ that admit  the analytic continuation into the half-plane $\bbC_\pm=\{z\in\bbC: \pm \Im z>0\}$ 
and satisfy  the estimate
$$
\sup_{\tau>0}\int_{-\infty}^\infty \abs{f(\lambda \pm i\tau)}^2 d\lambda<\infty.
$$
Let $P_\pm$ be the orthogonal projection
in $L^2(\bbR)$ onto $H^2_\pm(\bbR)$. 
The explicit formula for $P_\pm$ is 
\begin{equation}
(P_\pm f)(\lambda)=\pm \frac1{2\pi i}\lim_{\eps\to+0}
\int_{-\infty}^\infty  \frac{f(x)}{x-\lambda \mp i\eps}dx, 
\label{c4}
\end{equation}
where the limit exists for almost all $\lambda\in\bbR$; it also exists
  in $L^2(\bbR)$. The Hardy spaces  
of vector valued functions are defined quite similarly.
By a slight abuse of notation, we will use the same notation $P_\pm$ for 
the Hardy projections in these spaces;
they are defined by the same formula \eqref{c4}.

Now we are in a position to give the precise 

\begin{definition}\label{def.sho}
Let $\frakh$ be a Hilbert space (the case $\dim \frakh<\infty$ is not excluded), 
and let     an operator valued
function $\Xi\in L^\infty(\bbR;\calB(\frakh))$. The \emph{symmetrised Hankel operator} (SHO)  with  symbol $\Xi$ is the operator $M_\Xi$ given   on the space
$L^2(\bbR;\frakh)=L^2(\bbR)\otimes \frakh$ by formula
\eqref{b9a}.
\end{definition}

By definition, 
the operator $M_\Xi$ is bounded and self-adjoint. 
Since the operator $P_+ - P_-$  is unitary,
the simple identity 
$$
(P_+ -P_-)M_\Xi=-M_\Xi (P_+-P_-)
$$
shows that $M_\Xi$ is unitarily equivalent to $-M_\Xi$.
Thus the spectrum of $M_\Xi$ is symmetric with respect to 
the reflection: $\spec M_{\Xi} = -\spec M_{\Xi} $.

Many properties of SHOs are the same as those of the usual Hankel operators; we refer to the book \cite{Peller} for the theory of Hankel operators. From the definition it follows that
  $M_{\Xi_{1}}=M_{\Xi_{2}}$ if and only if the difference $\Xi_{1}-\Xi_{2}$ belongs to the Hardy class   $H_{+}^\infty(\bbR; \frakh)$   of operator valued functions that admit  a bounded analytic continuation in the upper half-plane.
If $\Xi\in C_0(\bbR;{\mathfrak S}_\infty(\frakh))$, 
then $M_\Xi$ is compact; see \cite[\S~2.4]{Peller}. 

We will consider SHOs with piecewise continuous symbols. 
Since functions of the class $H_{+}^\infty (\bbR; \frakh)$ 
cannot have jump discontinuities,  different symbols of such SHOs have the same jumps. Curiously, this fact follows from our results, and even more general statements about  functions in $H_{+}^\infty (\bbR; \frakh)$ follow from the results of \cite{Power} on the essential spectrum of Hankel operators.

\subsection{The model operator}\label{sec.b5bw}

Set 
\begin{equation}
\zeta(\lambda)
=
\frac{1}{\pi} 
\int_0^\infty\frac{\sin(\lambda t)}{2+t}dt.
\label{c6}\end{equation}
Obviously, $\zeta(\lambda)$ is a real and odd function. Since
$$
\zeta(\lambda)
=
\frac{1}{\pi} \Im \Big(e^{-2i\lambda} 
\int_\lambda^\infty\frac{e^{2i  x} }{x}d x\Big),\quad \lambda>0,
$$
the function  $\zeta\in C^\infty (\bbR\setminus\{0\})$
and $\zeta(\lambda)=O (|\lambda|^{-1})$ as $\abs{\lambda}\to\infty$. 
Moreover, the limits $\zeta(\pm0)$ exist, $\zeta(\pm0)=\pm 1/2$ and
\begin{equation}
\zeta'(\lambda) =O ( | \ln|\lambda||), \quad \lambda\to   0.
\label{c6x}\end{equation}

Let us now define
a particular SHO   in the space $L^2 (\bbR; \frakh)$ which can be explicitly diagonalized and plays a key role in our construction. For an arbitrary operator $K\in {\mathfrak S}_\infty(\frakh)$, we consider   the SHO $M_\Xi$ with the symbol $\Xi(\lambda)=\zeta(\lambda) K$ defined by formula \eqref{b9a}. 

Next, we define  the   weight function on $\bbR$  which vanishes logarithmically at the origin:
\begin{equation}
q_0(\lambda)
=
\begin{cases}
\abs{\log\abs{\lambda}}^{-1}& \text{ if $\abs{\lambda}<e^{-1}$, }
\\
1 & \text{ if $\abs{\lambda}\geq e^{-1}$. }
\end{cases}
\label{c7}
\end{equation}
We will also denote by $q_{0}$ the  operator of multiplication by $q_0 (\lambda) $
in the space $L^2 (\bbR; \frakh)$.

Our goal in this section is to study the spectral properties of the SHO $M_{\zeta K}$. First, we give the result for the scalar case. 

\begin{theorem}\label{lma.e4x}
Let     $M_\Xi$ be the SHO in $L^2(\bbR)$ with the symbol $\Xi= \zeta K$ 
where $\zeta $ is the function $\eqref{c6}$ and $K\in{\bbC}$, $K\neq 0$. Then:
\begin{enumerate}[\rm (i)]
\item
The operator $M_\Xi$ has the purely a.c.\  spectrum $[-|K|/2,|K|/2]$ of multiplicity one.
\item
For an arbitrary $\beta>1/2$, the  operator   $q_0^\beta  $  
is strongly $M_\Xi$-smooth    on   intervals  $(-|K|/2,0)$ and $(0, |K|/2)$ with any exponent  $\gamma <\beta-\frac12$, $\gamma\in (0,1]$.
\end{enumerate}
\end{theorem}

In the general case our result is stated as follows.

\begin{theorem}\label{lma.e4}
Let     $M_\Xi$ be the SHO in $L^2(\bbR;\frakh)$ with the symbol $\Xi= \zeta K$ 
where $\zeta $ is the function $\eqref{c6}$ and  $K\in {\mathfrak S}_\infty(\frakh)$. Then:
\begin{enumerate}[\rm (i)]
\item
Apart from the possible eigenvalue $0$, the operator $M_\Xi$ has the purely a.c.\  spectrum:
\begin{equation}
\spec_{{\rm ac}} M_\Xi =\bigcup_{n=1}^{\dim\frakh} [-\tfrac12 s_n(K) ,\tfrac12 s_n(K)]. 
\label{c7a}
\end{equation}
The operator $M_\Xi$ has eigenvalue $0$ if and only if either 
$\Ker K \neq \{0\}$ or $\Ker K^* \neq \{0\}$ $($or both$)$; if $0$ is an eigenvalue, then
its multiplicity in the spectrum of $M_\Xi$ is infinite. 
\item
For an arbitrary $\beta>1/2$, the  operator   $q_0^\beta  $  
is strongly $M_\Xi$-smooth  with any exponent  $\gamma <\beta-\frac12$, $\gamma\in (0,1]$, on all intervals $\delta\subset  [-\tfrac12 s_1(K) ,\tfrac12 s_1(K) ] $ which do not contain the points $0$ and $\pm \tfrac12 s_n(K)$.
\end{enumerate}
\end{theorem}
In \eqref{c7a} and in what follows we use the same convention as in formula \eqref{eq:a0}, i.e.\  the 
union is taken only over the intervals of positive length and each 
interval contributes multiplicity one to the a.c.\  spectrum of $M_\Xi$. 

The proofs of these results will be given in the rest of this section where 
a concrete diagonalization of the operator $M_\Xi$ will also be given.

\subsection{Hankel operators}\label{sec.HO}

 Let ${\sf H}$ be  an auxiliary Hilbert space, and let $\Sigma$ be  the operator    of multiplication by a bounded  operator valued function $\Sigma(\lambda):  {\sf H}\to  {\sf H}$ in the space $L^2 ({\bbR}; {\sf H})= L^2({\bbR}) \otimes {\sf H}$. We   set $( \mathcal{J}f)(\lambda)=f(-\lambda)$.
   Then  the  Hankel operator 
$\Gamma_{\Sigma}$ with symbol $\Sigma$  is defined  in the space $H^{2}_{+}({\bbR}; {\sf H})$ by the formula
\begin{equation}
\Gamma_{\Sigma}f=P_{+}\Sigma \mathcal{J}f.
\label{eq:hh}\end{equation}

Let $\Phi$ be the unitary Fourier transform in $L^2(\bbR)$ (or, more generally, in $L^2 ({\bbR}; {\sf H})$), 
$$
(\Phi f)(t)
=
\frac{1}{\sqrt{2\pi}}
\int_{-\infty}^\infty e^{-i\lambda t}f(\lambda)d\lambda. 
$$
Then $\Phi P_\pm \Phi^*=\chi_\pm$  
where $\chi_\pm$ is the operator of multiplication by the
characteristic function of $\bbR_\pm$ so that $\Phi: H^{2}_{\pm}({\bbR}) \to L^2(\bbR_{\pm})$.
Further, $\Phi \Sigma\Phi^*$ is the operator of convolution 
with the function $(2\pi)^{-1/2}  \wh{\Sigma}(t)$ where 
$\wh{\Sigma}=\Phi(\Sigma)$. It follows that the operator $\wh\Gamma_{\Sigma}= \Phi \Gamma_{\Sigma}\Phi^*: L^2(\bbR_{+}; {\sf H})\to L^2(\bbR_{+}; {\sf H})$ acts by the formula
\begin{equation}
(\wh\Gamma_{\Sigma}  f) (t)= (2\pi)^{-1/2}  \int_0^ \infty \wh{\Sigma}(t+s)f(s)ds.
\label{eq:hh1}\end{equation}

 A SHO $M_{\Xi}$ in the space $H^{2}_{+}({\bbR}; \frakh)$ is canonically  unitarily equivalent to a special  Hankel operator $\Gamma_{\Sigma}$ acting in the space $H^{2}_{+}(\bbR ; {\sf H})$ where ${\sf H}=\frakh\oplus\frakh$. Indeed, let   the unitary operator $J :L^2(\bbR;\frakh)\to H^{2}_{+}(\bbR  ) \otimes {\sf H}$ be defined by the formula 
$$
J f= (P_{+}f, P_{+} \mathcal{J}  f)^\top  .
$$
Then
\begin{equation}
J M_{\Xi}= \Gamma_{\Sigma} J \quad {\rm where}\quad
\Sigma(\lambda)=\begin{pmatrix}
0 &\Xi^* (\lambda)
\\
\Xi (-\lambda)    & 0
\end{pmatrix}: {\sf H}\to {\sf H}.
\label{eq:shh}\end{equation}
This can be checked by the following direct calculation.
Since $P_{+} \mathcal{J}=  \mathcal{J}P_-$, it follows from \eqref{b9a} that
\begin{align}
J M_{\Xi}f =& (P_{+}(P_- \Xi P_+ + P_+ \Xi^*P_- )f, \mathcal{J}P_-  (P_- \Xi P_+ + P_+ \Xi^*P_- )f)^\top
\nonumber\\
=& (P_{+}  \Xi^*P_- f, \mathcal{J}P_-    \Xi P_+  f)^\top.
\label{eq:shhx}\end{align}
Similarly,  it follows from \eqref{eq:hh} that
\begin{align}
\Gamma_{\Sigma} J f = P_{+} \Sigma
 (\mathcal{J } P_{+}    f,  P_-  f)^\top= ( P_{+} \Xi^*  P_-  f,  P_{+} (\mathcal{J }\Xi\mathcal{J }) \mathcal{J } P_{+}   f)^\top.
\label{eq:shhy}\end{align}
The r.h.s.\  in \eqref{eq:shhx} and \eqref{eq:shhy} coincide because
$ P_{+} (\mathcal{J }\Xi\mathcal{J }) \mathcal{J } P_{+} = P_{+}  \mathcal{J }\Xi  P_{+}= \mathcal{J }P_{-}   \Xi  P_{+}$.

In particular, a SHO $M_{\Xi}$ in   $L^{2} ({\bbR} )$ satisfies \eqref{eq:shh} with 
$\Sigma(\lambda):{\bbC}^2 \to {\bbC}^2$ and   $\Gamma_{\Sigma}$ acting  in   $H^{2}_{+}(\bbR ; {\bbC}^2)$.

\subsection{Mehler's formula}\label{sec.b5Me}

 Let us consider an integral operator
  $\mathcal M$ acting in the space $L^2(\bbR_+)$ 
by the formula
$$
(\mathcal M f)(t)=\frac1\pi \int_0^\infty  \frac{f(s)}{2+t+s} ds.
$$
Calculating the Fourier transform of the function \eqref{c6}, we find that
$$
 \wh{\zeta}(t)
=
-\frac{i}{\sqrt{2\pi}} \frac{\sign t }{2+\abs{t}},
\quad t\in\bbR,
$$
and hence (see \eqref{eq:hh1})
\begin{equation}
\mathcal M =\Phi   \Gamma_{2i\zeta}  \Phi ^*
\label{eq:mm}\end{equation}
 where $\Gamma_{2i\zeta}$ is the Hankel operator in $H^2_{+}(\bbR)$ defined by formula \eqref{eq:hh}.

We use the fact that the operator $\mathcal M$ can be explicitly  diagonalised.
Its diagonalisation  
is based on Mehler's formula (see \cite[Section~3.14]{BE}):
\begin{equation}
\frac1\pi
\int_0^\infty \frac{P_{-\frac12+i\tau}(1+s)}{2+t+s} d s
=
\frac1{\cosh(\pi \tau)}
P_{-\frac12+i\tau}(1+t),
\quad
t, \tau\in\bbR_+,
\label{f1}
\end{equation}
where $P_{\nu}$ is the Legendre function, see formulas \eqref{f16}, \eqref{eq:dd}
in the Appendix. A systematic approach to the proof of formulas of this type was suggested in \cite{Yafaev3}.
Recall (see \cite[\S 3.14]{BE}) that the Mehler-Fock transform 
$\Psi$ 
is defined by the formula
\begin{equation}
(\Psi f)(\tau) = \sqrt{\tau\tanh(\pi \tau)}  \int_0^\infty 
  P_{-\frac12+i\tau}(t+1)f(t)dt,
  \quad   
  \tau>0.
\label{f2}
\end{equation}

\begin{lemma}\label{Me}
The Mehler-Fock transform $\Psi$ maps $L^2(\bbR_+; dt)$ onto $L^2(\bbR_+ ; d\tau)$
and is unitary.  
It diagonalizes the operator $\mathcal M$:
\begin{equation}
(\Psi \mathcal M f)(\tau)
=
\frac1{\cosh(\pi \tau)} (\Psi f)(\tau).
\label{f3}
\end{equation}
\end{lemma}

The unitarity of  $\Psi$ is discussed in \cite{Yafaev3}, and \eqref{f3} is a consequence of  \eqref{f1}.

Thus, the map $\Psi$ reduces  the operator $\mathcal M$ to the operator of multiplication 
by the function $1/\cosh(\pi\tau)$ in the space $L^2(\bbR_+ ; d\tau)$. In particular, we see that $\mathcal M$ has the simple purely a.c.\  spectrum which coincides with the interval $[0,1]$.

Let us now define the operator   $ W=\Psi\chi_{+}\Phi: L^2(\bbR)\to L^2({\bbR}_{+})$. It is a unitary mapping of $H^2_{+}(\bbR)$ onto $L^2({\bbR}_{+})$ and $Wf=0$ if $f\in H^2_{-}(\bbR)$. 
It follows from definition \eqref{f2} that, formally, the   operator $W$ is given by the equality
\begin{equation}
(W  f)(\tau)
=
\sqrt{\tau\tanh(\pi \tau)} \int_{-\infty}^\infty w_{\tau}(\lambda)f(\lambda)d\lambda
\label{eq:Uex}\end{equation}
where
\begin{equation}
w_{\tau}(\lambda)= 
\frac1{\sqrt{2\pi}}
\int_0^\infty  
P_{-\frac12+i\tau}(t+1)e^{-i \lambda t} dt. 
\label{eq:Xx1}\end{equation}

Putting together formulas \eqref{eq:mm} and \eqref{f3}, we obtain the following result.

\begin{lemma}\label{Mex}
Let $\zeta$ be the function \eqref{c6} and
let the Hankel operator $\Gamma_{2i\zeta}$ in $H^2_{+}(\bbR)$  be defined by formula \eqref{eq:hh}.  Then
\begin{equation}
(W \Gamma_{2i\zeta} f)(\tau)
=
\frac1{\cosh(\pi \tau)} (W f)(\tau), 
\quad \tau>0.
\label{eq:Mex}\end{equation}
\end{lemma}

\subsection{The diagonalization of $M_{\Xi}$}\label{sec.b5bwt}

Let us   return to the
  SHO $M_\Xi$ with the symbol  $ \Xi (\lambda)=\zeta(\lambda)K$.
   For an arbitrary operator $K\in\calB(\frakh)$, a real odd function $\zeta$ and  $ \Xi (\lambda)=\zeta(\lambda)K$, relation \eqref{eq:shh} holds with
\begin{equation}
\Sigma(\lambda)= i\zeta(\lambda) {\sf K} \quad \text{where}\quad 
{\sf K} =\begin{pmatrix}
0 &- i     K^*
\\
i   K & 0
\end{pmatrix}.
\label{eq:fgx}\end{equation}

Let us define the unitary operator 
$
{\mathcal W}: L^2 ({\bbR})\otimes\frakh\to L^2 ({\bbR}_{+})\otimes{\sf H}
$
by the formula
\begin{equation}
{\mathcal W}f= ((W\otimes I)f , (W \mathcal{J} \otimes I)  f)^\top.
\label{eq:Uex1}
\end{equation}
Putting together formulas  \eqref{eq:shh} and \eqref{eq:Mex}, we obtain the following result.

\begin{lemma}\label{Me1}
Let $K$ be a bounded operator in a Hilbert space $\frakh$, let the function $\zeta$ be defined by formula \eqref{c6}, and let  $ \Xi (\lambda)=\zeta(\lambda)K$. Then 
\begin{equation}
{\mathcal W}  M_\Xi   {\mathcal W} ^*
=
\frac1{2\cosh(\pi \tau)}\otimes {\sf K}  
\label{e14x}\end{equation}
where the operator ${\sf K} $ is given by formula \eqref{eq:fgx}.
\end{lemma}

The operator $ (2\cosh(\pi \tau))^{-1}\otimes {\sf K}$ acting in the space 
$L^2({\bbR}_{+}; d\tau)\otimes {\sf H}$ is obviously unitarily equivalent 
to the operator $\mu \otimes {\sf K}$ acting in the space $L^2((0,1/2); d \mu)\otimes {\sf H}$. 
Thus the diagonalization of the operator $ M_\Xi $ reduces to that of the operator ${\sf K}  $.

Suppose now that $K$ is compact in $\frakh $. 
Then the operator ${\sf K}$ is   compact and self-adjoint   in the space ${\sf H} $. 
Let $s_{n}=s_n(K)$, $1\leq n\leq \dim\frakh$,   
be the singular values of $K$. 
It is easy to check the following result.

\begin{lemma}\label{K}
The nonzero spectrum of ${\sf K}$  
consists of the eigenvalues 
$$
\{\pm s_{n}(K): 1\leq n\leq \dim\frakh,\ \  s_n(K)>0\}.
$$ 
Moreover, $\Ker {\sf K}= \Ker K\oplus \Ker K^*$.
\end{lemma}
Thus, diagonalising the operator ${\sf K}$, we get the following result:

\begin{lemma}\label{Me1x}
Let $A=(2\cosh(\pi \tau))^{-1}\otimes {\sf K}$ in $L^2(\bbR_+)\otimes {\sf H}$. 
Then, apart from the possible eigenvalue at zero, the operator $A$ has the purely
a.c.\  spectrum: 
\begin{equation}
\spec_{\rm ac}  A =\bigcup_{n=1}^{\dim\frakh} [-\tfrac12 s_n(K) ,\tfrac12 s_n(K) ].
\label{c7b}
\end{equation}
The operator $A$ has eigenvalue zero  if and only if $\Ker {\sf K} \neq \{0\}$;
if zero is the eigenvalue of $A$, then it has infinite multiplicity. 
\end{lemma}
Putting together Lemmas~\ref{Me1} and \ref{Me1x}, 
we conclude the proof of Theorem~\ref{lma.e4}(i).

In the scalar case $\frakh ={\bbC}$, we have $K\in{\bbC}$, ${\sf H}={\bbC}^2 $ and 
the operator ${\sf K}$ in \eqref{eq:fgx} is the $2\times 2$ matrix. If $K\neq 0$, the spectrum of ${\sf K}$ consists of the two eigenvalues $\pm |K|$   and formula \eqref{c7b} for $\dim\frakh=1$ means that
$$
\spec_{\ac} A= [-\tfrac12 |K| ,\tfrac12 |K|] .
$$
In particular, the spectrum of $A$ is a.c.\   and simple.

\begin{remark}\label{eig}
It is easy to calculate eigenvectors of the operator $ {\sf  K}$ in terms of eigenvectors of the operator $\sqrt{K^* K}$. Indeed, if 
$\sqrt{K^* K}a_{n}=s_{n}a_{n}$, then  $b_{n}^{(\pm)}= (s_{n}a_{n},\pm i Ka_{n})^\top$ satisfy the equation
${\sf K}b_{n}^{(\pm)}=\pm s_{n}b_{n}^{(\pm)}$. In particular, in the case $\dim\frakh=1$ we have
$b^{(\pm)}=  (|K| ,\pm i K    )^\top$.
\end{remark}

\subsection{Smoothness with respect to the model operator}\label{sec.b5bwx}

  It remains   to check Theorems~\ref{lma.e4x}(ii) and \ref{lma.e4}(ii).  We use the diagonalization $\calF$ of the operator $M_{\Xi}$ defined by formulas \eqref{eq:Uex}, \eqref{eq:Uex1} and  \eqref{c7c}. 
  By Definition~\ref{stsm},  the proof of the strong $M_{\Xi}$-smoothness   of the operator $q^\beta$ requires     estimates on  the function $w_{\tau}(\lambda)$. They are collected in the following assertion. 
  
 
\begin{lemma}\label{lma.f5}
For $\tau>0$ and $\lambda \in\bbR$, $\lambda\neq 0$,
the function
$w_{\tau}(\lambda)$
 is differentiable in $\tau$.
If  $\delta\subset {\bbR}_{+}$ is a compact interval and $\tau\in\delta$, 
then for some constant $C=C(\delta)$ we have the estimates:
\begin{equation}
\abs{ w_{\tau}(\lambda)}
\leq 
C  \abs{\lambda}^{-1},
\quad
| \partial w_{\tau} (\lambda)/ \partial \tau |
\leq
C  \abs{\lambda}^{-1},
\quad 
\abs{\lambda}\geq 1/2,
\label{f4}\end{equation}
and 
\begin{equation}
\abs{ w_{\tau}(\lambda)  }
\leq 
C  \abs{\lambda}^{-1/2},
\quad 
| \partial w_{\tau} (\lambda)/ \partial \tau |
\leq
C  \abs{\lambda}^{-1/2}\big|\ln |\lambda | \big|,
\quad
\abs{\lambda}\leq 1/2, \quad \lambda\not=0.
\label{f5}
\end{equation}
\end{lemma}

The proof is quite elementary and is given in the Appendix. 
 The following assertion is an easy consequence of Lemma~\ref{lma.f5}.

\begin{lemma}\label{lma.e4y}
Let the operator $  W $ be defined by formula  \eqref{eq:Uex}. Set $W_{1}=W$, $W_{2}=W\mathcal{J}$.
Suppose that $\beta>1/2$ and $\gamma <\beta-\frac12$, $\gamma\leq 1$. 
Then, for $\tau$ and $\tau'$ in compact subintervals
of $\bbR_{+}$ and for all $f\in L^2(\bbR)$,  we have the estimates
\begin{equation}
\begin{split}
|(W_{1,2} q_0^\beta f)(\tau)|
&\leq 
C\norm{f} ,
\\
|( W_{1,2} q_0^\beta f)(\tau)-( W_{1,2}  q_0^\beta f)(\tau')| 
&\leq 
C
\abs{\tau- \tau'}^\gamma
\norm{f} .
\label{eq:uu}\end{split}
\end{equation}
\end{lemma}

\begin{proof}
It follows directly from Lemma~\ref{lma.f5} that  
\begin{align*}
\int_{-\infty}^\infty q_0^{2\beta}( \lambda)\abs{w_\tau(\pm\lambda)}^2 d\lambda
&\leq C,
\\
\int_{-\infty}^\infty q_0^{2\beta}( \lambda)\abs{w_\tau(\pm\lambda)-w_{\tau'}(\pm \lambda)}^2 d\lambda
&\leq C\abs{\tau-\tau'}^{2\gamma}.
\end{align*}
These estimates imply  estimates \eqref{eq:uu}.
\end{proof}  

Of course Lemma~\ref{lma.e4y} implies a similar statement about the operator 
$\mathcal W$ definied by formula \eqref{eq:Uex1}.

Let us denote by $\sf Y$ the unitary operator in $\sf H$ that diagonalises
$\sf K$:
\begin{equation}
{\sf Y K Y^*} 
=
\diag\{s_1,-s_1,s_2,-s_2,\dots\}, 
\quad s_n=s_n(K);
\label{c7c}
\end{equation}
the sequence in the r.h.s.\  of \eqref{c7c} has the same number of
zeros as $\dim\Ker\sf K$.
 
\begin{proof}[Proof of Theorem~{\rm \ref{lma.e4x}(ii)}] 
It follows from  formulas \eqref{e14x} and   \eqref{c7c} that   
the operator 
\begin{equation}
(I\otimes {\sf Y}) {\mathcal W}  M_\Xi   {\mathcal W}^* (I\otimes {\sf Y})^*
\label{c7d}
\end{equation}
acts in the space 
$L^2({\bbR_{+}}; {\bbC}^2)$ as the multiplication by the matrix valued function
$$
\frac{|K|}{2\cosh(\pi \tau)} 
 \begin{pmatrix}
1 & 0
\\
0   & -1
\end{pmatrix}.
$$
Making the changes of variables $\mu=\pm  2^{-1} |K| (\cosh(\pi \tau))^{-1}$, 
we see that it reduces to the multiplication by  $\mu$ in the space $L^2( -|K|/2, |K|/2)$. 
Therefore  estimates  \eqref{b4} for the diagonalization $\mathcal{F}=(I\otimes {\sf Y}) {\mathcal W}$
  of the operator  $M_\Xi$,  $Q=q_0^\beta$ and the operator $Z(\mu)$ defined by \eqref{eq:Z1} 
  are equivalent to estimates \eqref{eq:uu}. This concludes the proof.
\end{proof}
  
\begin{proof}[Proof of Theorem~{\rm \ref{lma.e4}(ii)}]   
In view of  \eqref{e14x} and   \eqref{c7c},   
the operator \eqref{c7d} acts in the space 
$L^2 ({\bbR}_{+})\otimes{\sf H}$
as the multiplication by the matrix valued function
\begin{equation}
(2\cosh(\pi \tau))^{-1}\otimes\diag \{s_{1}, -s_{1}, s_{2}, -s_{2}, \ldots\},\quad s_n=s_n(K).
\label{e4a}
\end{equation}
After the changes of variables $\mu =\pm  2^{-1} s_{n} (\cosh(\pi \tau))^{-1}$,   
the part of the operator \eqref{e4a} in the orthogonal complement to its kernel 
reduces to the multiplication by  $\mu$ in the space 
$\bigoplus_{n=1}^{\dim \frakh}  L^2( -s_n/2, s_n/2)$ (the sum is taken over $s_n>0$). 

We have to verify the $M_{\Xi}$-smoothness of the operator $Q=q_{0}^\beta$ on all intervals $\delta_{m}^{(+)}=(s_{m+1}/2,s_{m}/2)$ and $\delta_{m}^{(-)}=(-s_{m }/2,-s_{m+1}/2)$ where $s_{m+1}< s_{m}$. 
The operator $M_{\Xi}|_{\Ran E_{M_{\Xi}} (\delta_{m}^{(\pm)})}$ 
is unitarily equivalent to the operator of multiplication by $\mu$ in the space 
$L^2 (\delta_{m}^{(\pm)})\otimes{\bbC}^m$.  
Thus, estimates  \eqref{b4} for this diagonalization 
of the operator  $M_{\Xi}|_{\Ran E_{M_{\Xi}} (\delta_{m}^{(\pm)})}$    are again equivalent to estimates \eqref{eq:uu}. 
\end{proof}

\begin{remark}\label{gen}
In Sections~\ref{sec.ee} and \ref{sec.b5} we need the results which are formally more general than Theorem~\ref{lma.e4} but are its direct consequences. Let 
$$
\Xi(\lambda)=\zeta(\lambda-\lambda_0)K
$$
for some $\lambda_{0}\in \bbR$.  Then the first assertion of Theorem~\ref{lma.e4} is true for the SHO $M_{\Xi}$ with the symbol 
$$
\Xi(\lambda)=\zeta(\lambda-\lambda_0)K, \quad K\in\mathfrak{S}_{\infty} (\frakh).
$$ 
Set  
\begin{equation}
q(\lambda)=\prod_{\ell=1}^L q_0(\lambda-\lambda_\ell)
\label{e4}
\end{equation}
where   the function $ q_0(\lambda)$ is defined by formula \eqref{c7}. 
Then  the  assertion of Theorem~\ref{lma.e4} about strong $M_\Xi$-smoothness remains true for the operator of multiplication by the function $q^\beta (\lambda)$ if $\lambda_\ell=\lambda_{0}$ for one of $\ell$.
\end{remark}

\section{Compactness of sandwiched Hankel operators}\label{sand}

In this  section we prepare   auxiliary statements about the compactness
of some Hankel type operators appearing in our construction.  

\subsection{Muckenhoupt weights}\label{sec.Muck}

Recall that a function $v\in L^1_\loc(\bbR)$, $v\geq0$, such that $v^{-1} \in L^1_\loc(\bbR)$,
is called a \emph{Muckenhoupt weight} if 
\begin{equation}
\sup_\Lambda 
\frac1{\abs{\Lambda}}\int_\Lambda v(\lambda)d \lambda 
\cdot 
\frac1{\abs{\Lambda}}\int_\Lambda v(\lambda)^{-1}d \lambda
<\infty,
\label{d20a}
\end{equation}
where the supremum is taken over all bounded intervals $\Lambda\subset\bbR$, 
and $\abs{\Lambda}$ denotes the length of $\Lambda$. 
It is a classical result  of \cite{Muck}  that the operators $P_\pm$ are bounded in 
$L^2(\bbR;v(\lambda)d\lambda)$ iff $v$ is a Muckenhoupt weight. 
Thus, for a Muckenhoupt weight $v$, the operators 
$v^{1/2} P_\pm v^{-1/2}$ are bounded. 
Of course, the same result is true for the operators $P_\pm$ in the vector-valued
space $L^2(\bbR,\frakh)$, if $v$ is a scalar function satisfying \eqref{d20a}.

We consider operators acting in the space $L^2(\bbR;\frakh)$ where $\frakh$ is an auxiliary Hilbert space. Let the function $q(\lambda)$ be defined by formulas \eqref{c7} and \eqref{e4}.   We use the same notation $q$ for the operator of multiplication by the function $q(\lambda)$  acting  in different spaces (for example, in $L^2(\bbR)$ and $L^2(\bbR;\frakh)$).
   It is easy to check that, for any $\beta\in\bbR$, the function $q^{2\beta} (\lambda)$  is a Muckenhoupt weight.  This yields the following result.

\begin{proposition}\label{Muck}
For any $\beta\in  \bbR$, the operators $q^\beta P_\pm q^{-\beta}$ 
  and hence  $q^\beta M_\Xi q^{-\beta}$ 
are bounded in $L^2(\bbR;\frakh)$ 
for all $\Xi\in L^\infty(\bbR;\calB(\frakh))$. 
\end{proposition}

\subsection{Singular weight functions}\label{sec.e5}

Recall (see \cite{Peller}) that for $\Xi\in C_0(\bbR;{\mathfrak S}_\infty(\frakh))$, we have
\begin{equation}
\Xi P_+ - P_{+}\Xi= P_{-} \Xi P_{+} - P_{+} \Xi P_{-} \in {\mathfrak S}_\infty .
\label{eq:Pe}
\end{equation}
We shall see that the operator \eqref{eq:Pe}
remains compact even after being sandwiched between singular weights $q^{-\beta}$ provided the symbol $\Xi$ is   logarithmically H\"older  continuous at  the singular points of $q^{-\beta}$. Let us first consider sufficiently smooth symbols
with compact supports.
  
\begin{lemma}\label{lma.d1A}
Let $\Xi\in C (\bbR;{\mathfrak S}_\infty(\frakh))$ and $\Xi\in C^1$ in some neighbourhoods 
of the singular points $\lambda_{1}, \ldots, \lambda_{L}$ of the weight $q^{-\beta}$.
Suppose also that $\Xi$ has compact support.
Then the  operator
\begin{equation}
{\sf G}=\X^{-\beta}(\Xi P_+ -P_+\Xi)\X^{-\beta} 
\label{d21}
\end{equation}
is compact
in the space $L^2(\bbR;\frakh)$ for all $\beta\in  \bbR$. 
\end{lemma}

\begin{proof}
Let $A$ be an operator in $L^2(\bbR;\frakh)$   with the integral kernel $a(x,y)$. We proceed from the obvious estimate
\begin{equation}
\| A\|^2\leq \int_{-\infty}^\infty \int_{-\infty}^\infty \|a(x,y)\|^2_{\calB(\frakh)} dxdy.
\label{eq:OE}\end{equation}

According to formula \eqref{c4} the integral kernel of the operator $K=\Xi P_+ -P_+\Xi$ equals
$$
k(x,y)= - \frac{1}{2\pi i}\frac{\Xi(x)-\Xi(y)}{x-y}.
$$
It is a continuous function with values in ${\mathfrak S}_\infty(\frakh)$, and it satisfies the bound 
\begin{equation}
\norm{k(x,y)}_{\calB(\frakh)}
\leq
C (1+\abs{x})^{-1}(1+\abs{y})^{-1}.
\label{eq:OE1}\end{equation}

Let $\chi_{\varepsilon} (x)$ be the characteristic function of the set
\begin{equation}
Q_{\varepsilon}=\{x\in\bbR: |x-\lambda_{\ell}  |>\varepsilon,  \; \forall \ell=1,\ldots, L \quad{\rm and}\quad |x|<\varepsilon^{-1}\}.
\label{eq:qe}\end{equation}
Set ${\sf G}_{\varepsilon}  =\chi_{\varepsilon}   {\sf G} \chi_{\varepsilon}$. 
It follows from estimates \eqref{eq:OE} and \eqref{eq:OE1} that
$$
\| {\sf G}- {\sf G}_{\varepsilon} \|^2\leq C \int_{-\infty}^\infty \int_{-\infty}^\infty
|1-\chi_{\varepsilon} (x)|^2 q(x)^{-2\beta}  (1+\abs{x})^{-2}(1+\abs{y})^{-2}
 q(y)^{-2\beta} dxdy
$$
 tends to zero as $\varepsilon\to 0$.
 
 It remains to show that ${\sf G}_{\varepsilon}   $ is compact in $L^2(\bbR;\frakh)$
  for all $\varepsilon>0$. Observe that the integral kernel $  g_{\varepsilon} (x,y) $ of this operator 
 is a continuous operator valued function   on $Q_{\varepsilon}\times Q_{\varepsilon}$ with values in  $\mathfrak{S}_{\infty}(\frakh) $. It is almost obvious that such operators are compact. Indeed, let us split the interval $(-\varepsilon^{-1}, \varepsilon^{-1})$ in $N$ equal intervals $(x_{1},x_{2}),\ldots, (x_{N},x_{N+1})$, and let $\wt\chi_{n} (x)$ be the characteristic function of the interval $(x_{n},x_{n+1})$. Set
$$
 g_{\varepsilon,N} (x,y)=\sum_{n,m=1}^N g_{\varepsilon} (x_{n},y_{m})\wt\chi_{n} (x)\wt\chi_m (y).
$$
 Clearly, the operator ${\sf G}_{\varepsilon,N} $ with such kernel is compact in $L^2(\bbR;\frakh)$. Since $  g_{\varepsilon} (x,y) $ 
 is a continuous   function, we have 
$$
\lim_{N\to\infty} \sup_{x\in Q_{\varepsilon}, y\in Q_\varepsilon} \| g_{\varepsilon} (x,y) - g_{\varepsilon,N} (x,y)\|_{\calB(\frakh)} =0.
$$
Therefore
$
 \|{\sf G}_{\varepsilon}   - {\sf G}_{\varepsilon,N}  \|=0
$ as $N\to\infty$
according again to \eqref{eq:OE}. 
\end{proof}


Next we pass to the general case.

\begin{lemma}\label{lma.d1}
Let $\Xi\in C_0(\bbR;{\mathfrak S}_\infty(\frakh))$ be such that
\begin{equation}
\lim_{\lambda\to\lambda_\ell}
\norm{\Xi(\lambda)-\Xi(\lambda_\ell)}_{\calB(\frakh)}
q(\lambda)^{-2\beta}=0
\label{d0}
\end{equation}
for all $\ell=1,\dots,L$ and some $\beta>0$. 
Then the  operator \eqref{d21} is compact in $L^2(\bbR;\frakh)$. 
\end{lemma}

\begin{proof}
Let us reduce the question to the case 
\begin{equation}
\Xi(\lambda_1)=\dots=\Xi(\lambda_L)=0.
\label{d22}
\end{equation}
For $\ell=1,\dots,L$, let $\rho_\ell\in C_0^\infty(\bbR)$ 
be such that $\rho_\ell(\lambda_k)=\delta_{\ell k}$. 
Consider the function 
$$
\Xi_0(\lambda)
=
\sum_{\ell=1}^L \Xi(\lambda_\ell)\rho_\ell(\lambda).
$$
The   operator $\X^{-\beta}(P_+\Xi_0- \Xi_0 P_+)\X^{-\beta}$
is compact according to Lemma~\ref{lma.d1A}. 
Thus, we may check the compactness of the operator  \eqref{d21} with $\Xi-\Xi_0$ in place of $\Xi$. 
The function $\Xi-\Xi_0$ satisfies the condition $\Xi(\lambda_\ell)-\Xi_0(\lambda_\ell)=0$ for all $\ell$.  To simplify our notation, we will 
assume that $\Xi$ already satisfies \eqref{d22} and hence
\begin{equation}
\lim_{\lambda\to\lambda_\ell}
\Big(\norm{\Xi(\lambda) }_{\calB(\frakh)}
q(\lambda)^{-2\beta}\Big)=0.
\label{d0x}\end{equation}

  Let $\sigma_{\varepsilon} \in C_0^\infty(\bbR)$,  $\sigma_{\varepsilon}(\lambda)=0$ in neighbourhoods of all points $\lambda_{1},\ldots, \lambda_{L}$, $0\leq \sigma_{\varepsilon}(\lambda)\leq 1$ and  $\sigma_{\varepsilon}(\lambda)=1$ for $\lambda\in Q_{\varepsilon}$ (see \eqref{eq:qe}).
According to Lemma~\ref{lma.d1A} applied to $\Xi \sigma_{\varepsilon}$ in place of $\Xi$, the operators
\begin{equation}
\X^{-\beta}(\Xi\sigma_\eps P_+ - P_+\Xi \sigma_\eps)\X^{-\beta}
\label{d25} 
\end{equation}
are compact. Observe that
\begin{multline*}
\norm{
\X^{-\beta}\Xi(\sigma_\eps -1) P_+ \X^{-\beta}} 
\leq \norm{
\X^{-\beta}\Xi(\sigma_\eps -1)   \X^{-\beta}} \norm{\X^\beta P_+ \X^{-\beta}}
\\
\leq \sup_{\lambda\in\bbR}
\Big(\norm{\Xi(\lambda)}_{\calB(\frakh)}q(\lambda)^{-2\beta}\abs{1-\sigma_\eps(\lambda)}\Big)
 \norm{\X^\beta P_+ \X^{-\beta}}.
\end{multline*}
Since $\X^\beta P_+ \X^{-\beta}\in\mathcal{B}$,
 it follows from \eqref{d0x} and the condition $\| \Xi (\lambda)\|\to 0$ as $|\lambda|\to \infty$ that the r.h.s.\  here tends to $0$ as $\varepsilon\to 0$.  Thus
 the operators \eqref{d25} approximate the operator \eqref{d21} 
in the operator norm. This proves the compactness of \eqref{d21}. 
\end{proof}

\begin{corollary}\label{lma.e2c}
Under the hypothesis of Lemma~$\ref{lma.d1}$,
the operators $\X^{-\beta}P_\pm \Xi P_\mp \X^{-\beta}$ and therefore
 $\X^{-\beta}M_\Xi \X^{-\beta}$  are compact
in $L^2(\bbR; \frakh)$. 
\end{corollary}

\begin{proof}
Consider, for example, the operator
$$
\X^{-\beta}P_- \Xi P_+ \X^{-\beta}
=
\X^{-\beta}P_- (\Xi P_+ -P_+\Xi)\X^{-\beta}
=
(\X^{-\beta}P_- \X^\beta)\X^{-\beta}(\Xi P_+-P_+\Xi)\X^{-\beta}. 
$$
It suffices to use   the fact that the operator $\X^{-\beta}P_- \X^\beta$ is bounded and the operator \eqref{d21} is compact.
\end{proof} 

\begin{remark}\label{rem.e2c}
Lemma~\ref{lma.d1} and  Corollary~\ref{lma.e2c} can be extended in an obvious way to operators acting from a space  $L^2(\bbR; \frakh_{1})$ to a space $L^2(\bbR; \frakh_2)$ where $\frakh_1\neq\frakh_2$.
\end{remark}

\subsection{Disjoint singularities}\label{disjoint}

The following assertion will allow us to separate contributions of different jumps of the function $\Xi$. It suffices   to consider scalar valued symbols, that is, SHOs in the space $L^2({\bbR})$.

\begin{lemma}\label{lma.e2}
Let $\zeta$ be the function defined by \eqref{c6} and $\zeta_{j}(\lambda)= \zeta(\lambda-\lambda_{j})$, $j=1,2$. Suppose that $\lambda_{1}\neq \lambda_{2}$.  
Then, for all $\beta\in\bbR$, the operators
\begin{equation}
q^{-\beta} P_\pm \zeta_{1} P_\mp \zeta_{2} P_\pm q^{-\beta}
\label{eq:MM}\end{equation}
are compact  and therefore the operators 
$q^{-\beta}M_{\zeta_{1}}M_{\zeta_{2}}q^{-\beta}$
are also compact in $L^2 (\bbR)$. 
\end{lemma}

\begin{proof}
Choose functions $\rho_j\in C_0^\infty(\bbR)$, $j=1,2$, 
such that $\rho_j(\lambda)=1$ in a neighbourhood of the point $\lambda_j$ 
and  
\begin{equation}
\dist(\supp\rho_1,\supp\rho_2)>0.
\label{eq:MM2}\end{equation}
Set $\wt \rho_j=1-\rho_j$. Since
$\wt \rho_j\zeta_j\in C^\infty(\bbR )$ and
 $\wt \rho_j (\lambda) \zeta_j (\lambda) \to 0 $ as $|\lambda|\to \infty$,  it follows from Corollary~\ref{lma.e2c} that
\begin{equation}
q^{-\beta} P_\pm \zeta_{j} \wt \rho_j P_\mp   q^{-\beta}\in \mathfrak S_{\infty}, \quad j=1,2.
\label{eq:MM1}\end{equation}

Let us consider  the operator \eqref{eq:MM}, for example, for the upper signs. Writing $P_{-}$ as
$$
  P_-  =
  \rho_1 P_-   \rho_2  +  \wt \rho_1 P_-  +
 \rho_1   P_-  \wt\rho_2  ,
$$
we see that 
\begin{multline}
q^{-\beta} P_+ \zeta_1 P_- \zeta_2 P_+ q^{-\beta}
=
(q^{-\beta} P_+ q^{\beta}) (q^{-\beta} \zeta_1 \rho_1 P_- \zeta_2 \rho_2 q^{-\beta})(q^{\beta} P_+ q^{-\beta})
\\
+
(q^{-\beta} P_+ \zeta_1 \wt \rho_1 P_- q^{-\beta}) \zeta_2  (q^{\beta} P_+ q^{-\beta})
+
( q^{-\beta} P_+ q^{\beta}) \zeta_1 \rho_1 ( q^{-\beta} P_- \zeta_2 \wt\rho_2 P_+ q^{-\beta}).
\label{e5}\end{multline}
Recall that, by Proposition~\ref{Muck}, the operator $q^{-\beta} P_+ q^{\beta}$ is bounded.
Therefore to show that the first term in the r.h.s.\  of \eqref{e5} is compact, it suffices to check that 
$$
\psi_{1} P_{-} \psi_{2}\in \mathfrak S_{\infty}
$$
where $\psi_{j}= q^{-\beta} \zeta_j \rho_j \in L^2 (\bbR)$.  Using formula \eqref{c4} for the integral kernel of $P_{-}$ and condition \eqref{eq:MM2}, we find that the integral  kernel of the operator $\psi_{1} P_{-} \psi_{2}$ equals
$$
(2\pi i)^{-1} \psi_{1}(x) (x-y)^{-1} \psi_{2}(y).
$$
This function belongs to $L^2 (\bbR^2; dxdy)$ so that the operator $
\psi_{1} P_{-} \psi_{2}$ is Hilbert-Schmidt.

The second and third terms   in the r.h.s.\  of \eqref{e5} are compact because, by \eqref{eq:MM1}, the operators $q^{-\beta} P_+ \zeta_1 \wt \rho_1 P_- q^{-\beta}$ and $q^{-\beta} P_- \zeta_2 \wt\rho_2 P_+ q^{-\beta}$ are compact.

In view of definition \eqref{b9a} the compactness of operators \eqref{eq:MM} implies the
compactness of $q^{-\beta}M_{\zeta_{1}}M_{\zeta_{2}}q^{-\beta}$.
\end{proof}

\section{Spectral and scattering theory of symmetrised Hankel operators}\label{sec.ee}

Here we construct  the spectral theory  of SHOs with  piecewise continuous symbols.  This theory may be interesting
in its own right but, most importantly for us, it serves as a model for 
the spectral theory of the operators $\DD_\theta$.

\subsection{Main results for SHOs}\label{sec.e4}

Let $\Xi\in L^\infty (\bbR;{\mathfrak S}_\infty(\frakh))$ and let $M_\Xi$ be the SHO in 
$\calH=L^2(\bbR; \frakh)$ introduced in Definition~\ref{def.sho}.
We consider the case of
piecewise continuous operator valued symbols $\Xi$ with jump discontinuities at the points $\lambda_{1},\ldots, \lambda_{L}$. 
We use the notation 
\begin{equation}
K_\ell = \Xi(\lambda_\ell+0)- \Xi(\lambda_\ell-0)
\label{eq:jump}\end{equation}
for the jumps  of $\Xi$. These jumps are compact operators in $\frakh$. 
We denote by 
$s_n(K_\ell)$, $1\leq n\leq \dim \frakh$,
the sequence of singular values of the  operators $K_\ell $. Recall that the function $q(\lambda)$ was defined by formulas \eqref{c7} and \eqref{e4}.

\begin{theorem}\label{thm.f8}
Let a symbol $\Xi (\lambda)$ with values in $ {\mathfrak S}_\infty(\frakh)$ be a norm-continuous function of $\lambda$ apart from some jump discontinuities at finitely many points 
$ \lambda_1, \ldots, \lambda_L$. 
Assume that for each $\ell=1,\dots,L$, the symbol $\Xi$ satisfies
the logarithmic regularity condition 
\begin{equation}
\norm{\Xi(\lambda_\ell\pm\eps)- \Xi(\lambda_\ell\pm0)}
=
O(\abs{\log \eps}^{-\beta_0}), \quad \eps\to+0,
\label{e8}
\end{equation}
with an exponent $\beta_0>2$, and assume 
\begin{equation}
\lim_{|\lambda|\to \infty} \| \Xi(\lambda)  \|=0.
\label{e11}
\end{equation} 
Then:

\begin{enumerate}[\rm (i)]

\item 
The a.c.\  spectrum of the  operator $M_\Xi$  consists of the union of the intervals: 
\begin{equation}
\spec_{\rm ac} M_\Xi 
= 
\bigcup_{\ell=1}^L
\bigcup_{n=1}^{\dim \frakh}
[-\tfrac12 s_n(K_{\ell}),  \tfrac12 s_n(K_{\ell}) ].
\label{f18}
\end{equation}

\item
The singular continuous spectrum of $M_\Xi$ is empty. 

\item
The eigenvalues of $M_\Xi$ can accumulate only to $0$ and to the points 
$\pm \tfrac12 s_n(K_\ell)$. 
All eigenvalues of $M_\Xi$, distinct from 
$0$ and from $\pm \tfrac12 s_n(K_{\ell})$, 
have finite multiplicities.

\item
The operator valued function $q^\beta(M_\Xi-zI)^{-1}q^\beta$, 
$\beta>1 $, is  continuous in $z$
for $\pm\Im z\geq0$ 
if $z$ is separated away from the thresholds $0$,  $\pm  \tfrac12 s_n(K_{\ell})$ and  from the eigenvalues of $M_\Xi$.
\end{enumerate}
\end{theorem}

Note that already the case $\dim \frakh=1$ is non-trivial. 
In fact, this case contains most of  the essential difficulties
and the generalisation  to the case 
$\dim \frakh>1$ and even to $\dim\frakh= \infty$ is almost automatic.

Our next goal is to describe the structure of the a.c.\  subspace of the SHO $M_\Xi$.
We use tools of the scattering theory which also give   information on the  behaviour of the unitary group $\exp(-i M_\Xi t)f$ as $t\to \pm \infty$ for $f$ in the a.c.\  subspace of the operator $M_\Xi$. The following assertion shows that, for large $|t|$, the function $\exp(-i M_\Xi t)f$ ``lives" only in neighbourhoods of the singular points $\lambda_{1},\ldots, \lambda_{L}$.

\begin{lemma}\label{as}
Let $\Xi$ satisfy the assumptions of Theorem~$\ref{thm.f8}$. Let $Q$ be a closed set such that $Q\cap\{\lambda_{1},\ldots, \lambda_{L}\}=\varnothing$ and let $\chi_{Q}$ be the characteristic function of $Q$. 
 Then the operator $\chi_{Q} M_\Xi$ is compact and  
 \begin{equation}
 \slim_{|t|\to \infty}\chi_{Q} \exp(-i M_{\Xi}t) P^{({\rm ac})}_{M_{\Xi}}  =0.
\label{eq:comp}
\end{equation} 
  \end{lemma}

\begin{proof}
Choose a function $\omega\in C^\infty (\bbR)$ such that  $\omega(\lambda)=0$ in a neighbourhood 
of the set $\{\lambda_{1},\ldots, \lambda_{L}\}$, $\omega(\lambda)=1$ 
away from some larger neighbourhood of this set and $\chi_{Q}=\chi_{Q}\omega$.  
We have
\begin{equation}
\omega P_{\pm} \Xi P_{\mp}= (\omega P_{\pm}-  P_{\pm}\omega) \Xi P_{\mp}+ P_{\pm} \omega \Xi P_{\mp}.
\label{eq:compX}
\end{equation}
The first term in r.h.s.\  is compact because   
$$
\omega P_{\pm}-  P_{\pm}\omega 
= 
(\omega-1) P_{\pm}-  P_{\pm}(\omega-1)  \in{\mathfrak S}_\infty
$$
according to \eqref{eq:Pe}. The second term in r.h.s.\  of \eqref{eq:compX} is compact because 
$\omega \Xi \in C_{0}(\bbR; {\mathfrak S}_\infty ({\mathfrak h}))$. 
Thus $ \omega M_\Xi \in{\mathfrak S}_\infty$ and hence $ \chi_{Q} M_\Xi \in{\mathfrak S}_\infty$.
It follows that \eqref{eq:comp} holds on all elements of the form $f=M_\Xi g$ and therefore 
it holds for all $f\in L^2(\bbR,\frakh)$. 
\end{proof}

Note that condition \eqref{e8} and even the existence of the limits $\Xi(\lambda_\ell\pm0)$ were not used in the proof. 

Lemma~\ref{as} shows that it is natural to construct model operators
for each jump of $\Xi$. Set
\begin{equation}
\Xi_\ell(\lambda)
=
\zeta(\lambda-\lambda_\ell) K_\ell , \quad \ell=1,\dots,L,
\label{e9}
\end{equation}
where $\zeta$ is given by \eqref{c6} and $K_\ell$ is given by  \eqref{eq:jump}. As the model operator for  the point $\lambda_\ell$ we choose  the SHO $M_{\Xi_\ell}$.
Note that each of the symbols $\Xi_\ell$, $\ell\geq 1$, has only one jump at the point $\lambda_{\ell}$  and 
the spectral structure of $M_{\Xi_\ell}$ is described in Section~\ref{sec.e}.

\begin{theorem}\label{thm.f10}
Let the assumptions of Theorem~$\ref{thm.f8}$ hold true, let 
$\Xi_\ell$ be as defined in \eqref{e9}, and let $M_{\Xi_\ell}$ be as 
defined in \eqref{b9a}. Then:
\begin{enumerate}[\rm (i)]

\item
The wave operators  
$$
W_\pm(M_\Xi, M_{\Xi_\ell})
=
\slim_{t\to\pm\infty} e^{i M_\Xi t}  e^{-i M_{\Xi_\ell} t} P^{({\rm ac})}_{M_{\Xi_\ell} },
\quad 
\ell=1,\dots,L,
$$
exist and enjoy the intertwining property
$$
M_\Xi W_\pm(M_\Xi ,M_{\Xi_\ell})=W_\pm(M_\Xi,M_{\Xi_\ell}) M_{\Xi_\ell}.
$$
These operators are isometric and their ranges are orthogonal to each other, i.e. 
$$
\Ran W_\pm(M_\Xi,M_{\Xi_j})\perp \Ran W_\pm(M_\Xi,M_{\Xi_k}), 
\quad j\not=k.
$$

\item
The asymptotic  completeness holds:
$$
\Ran W_\pm(M_\Xi,M_{\Xi_1})
\oplus\cdots\oplus
\Ran W_\pm(M_\Xi,M_{\Xi_L})
=
\calH_{M_\Xi}^{({\rm ac})} .
$$
 
\end{enumerate}
\end{theorem}

Instead of \eqref{e11}, it suffices to  assume that $\Xi(\lambda)$ has a finite limit $\xi_{\infty}$ as $|\lambda|\to \infty$, that is,
\begin{equation}
\lim_{|\lambda|\to \infty} \| \Xi(\lambda) -\xi_{\infty} \|=0.
\label{e11w}
\end{equation}
Indeed, set $\wt{\Xi}(\lambda)=\Xi(\lambda) -\xi_{\infty} $. Then     Theorems~\ref{thm.f8} and \ref{thm.f10} can be applied to the operator $M_{\wt{\Xi}}$. Since $M_{\Xi}=M_{\wt{\Xi}}$, this yields all required results about  the operator $M_{ {\Xi}}$. However, assumption \eqref{e11w} can also be relaxed ---
see Remark~\ref{rmk}.

\subsection{Proofs of Theorems~\ref{thm.f8} and \ref{thm.f10}}\label{sec.e6}

 Theorems~\ref{thm.f8} and \ref{thm.f10} will be deduced from Propositions~\ref{pr.c1} and \ref{pr.c2}, respectively. We will check the hypotheses of  Proposition~\ref{pr.c1} for $A =M_\Xi$,  $A_\ell=M_{\Xi_\ell}$, $\ell=1,\dots, L$, and $A_\infty=M_{\Xi_\infty}$ where  the   symbol $\Xi_\infty$ is given by 
$$
\Xi_\infty (\lambda) = \Xi (\lambda)-\sum_{\ell=1}^L \Xi_\ell (\lambda).
$$
Then equality  \eqref{c3} is satisfied.
 As $\delta$, we choose an arbitrary open interval  $\delta\subset \bbR$ which does not contain  the points  $0$, $\pm\frac12 s_n( K_\ell)$, $1\leq n\leq\dim\frakh$, $\ell=1,\dots, L$. 
 
Set  $X=q^\beta$. The parameter $\beta>0$ is chosen
     sufficiently large to guarantee strong $A_\ell$-smoothness of $X$. On the other hand, it should   be sufficiently small to ensure  conditions \eqref{c1} and \eqref{c2}. To be more precise, we suppose that $1<\beta<\beta_0/2$, where
     $\beta_0$ is the exponent from \eqref{e8}.
     
      By Theorem~\ref{lma.e4},  the operator $q^\beta$ is strongly $A_\ell$-smooth on $\delta$ with any exponent $\gamma <\beta-1/2$  (and of course $\gamma\leq 1$) which allows us to choose $\gamma>1/2$. 

The operators $q^{-\beta} M_{\Xi_j }M_{\Xi_k}q^{-\beta}$, 
$1\leq j ,k\leq L$, $j \not=k$,  are compact by Lemma~\ref{lma.e2}. This yields condition \eqref{c1}.
 
Next,     $\Xi_\infty$ is a continuous function because   the functions $\Xi (\lambda)$ 
and   $ \Xi_\ell (\lambda)$ have the same jumps at all points $\lambda=\lambda_\ell$ and the functions
 $ \Xi_j (\lambda)$ are continuous at   $\lambda=\lambda_\ell$ if $j\neq\ell$.
Moreover, according to \eqref{e8} the function
$\Xi_\infty$ satisfies  assumption 
\eqref{d0} with any $\beta<\beta_0/2$. 
 By Corollary~\ref{lma.e2c}, it follows that the operator $q^{-\beta}M_{\Xi_\infty} q^{-\beta}$  is compact.

Finally, the operators $q^\beta M_{\Xi_\ell}q^{-\beta}$ 
are bounded according to Proposition~\ref{Muck}. 
Thus, all  the assumptions of  Proposition~\ref{pr.c1} are satisfied.

   On each   interval $\delta$, every statement of Proposition~\ref{pr.c1} yields the corresponding statement of  Theorem~\ref{thm.f8} about the operator $M_{\Sigma}$. Using that $\delta$ is arbitrary, we obtain the same statements on the whole line $\bbR$  with the points $0 $ and $\pm\frac12 s_n(K_\ell) $ removed. This concludes
the proof of Theorem~\ref{thm.f8}. 

Similarly, all conclusions of Proposition~\ref{pr.c2} are true for the wave operators
$W_\pm(M_\Xi,M_{\Xi_\ell}; \delta )$. Let us now use the fact that linear combinations of all elements $f$ such that $f\in \Ran E_{M_{\Xi_\ell}}(\delta)$ for some admissible $\delta$ are dense in $\calH^{(\ac)}_{M_{\Xi_\ell}}$. Therefore all statements of Proposition~\ref{pr.c2} about  the wave operators
$W_\pm(M_\Xi,M_{\Xi_\ell}; \delta )$ yield the corresponding    statements of 
Theorem~\ref{thm.f10} about the wave operators
$W_\pm(M_\Xi,M_{\Xi_\ell} )$.
 \qed 

 \begin{remark}\label{rmk1}
 If $\Xi$ has only one jump (i.e.\  $L=1$), the proof simplifies considerably. In this case it suffices to use the usual smooth scheme of scattering scattering theory for the pair $M_{\Xi_{1}}$, $M_{\Xi}$. 
 Lemma~\ref{lma.e2} is not required either.
\end{remark}

\subsection{SHOs on the circle}\label{sec.e8}

Let us briefly discuss the analogues of Theorems~\ref{thm.f8} and \ref{thm.f10}
for SHOs on the unit circle.
Let $H^2_+(\bbT; \frakh)\subset L^2(\bbT; \frakh)$ be the Hardy space of $\frakh$-valued
functions analytic in the unit disc, let $H^2_-(\bbT; \frakh)$ be the orthogonal 
complement of $H^2_+(\bbT; \frakh)$ in $L^2(\bbT; \frakh)$, and let 
$\bbP_\pm$ be the orthogonal projection in $L^2(\bbT; \frakh)$ onto $H^2_\pm(\bbT; \frakh)$. 
For $\Psi\in L^\infty(\bbT;\calB(\frakh))$ we define, similarly to \eqref{b9a}, 
the SHO  
$$
\bM_\Psi=\bbP_- \Psi\bbP_+  + \bbP_+ \Psi^* \bbP_-
\quad \text{in $L^2(\bbT;\frakh)$.}
$$
The spectral analysis of $\bM_\Psi$ can be obtained through a unitary map
from $L^2(\bbT;\frakh)$ onto $L^2(\bbR; \frakh)$. 
Indeed, the map
$$
\bbR\ni\lambda\mapsto\frac{\lambda-i}{\lambda+i}=\mu\in\bbT
$$
generates the unitary operator
\begin{equation}
\begin{split}
\calU & :  L^2(\bbR; \frakh ) \to L^2(\bbT; \frakh),
\\
(\calU f)(\mu)&=2i\sqrt{\pi}(1-\mu)^{-1}f(i\tfrac{1+\mu}{1-\mu}).
\end{split}
\label{e20}
\end{equation}
This operator transforms the SHO $M_\Xi$ in $L^2(\bbR; \frakh)$ into
the SHO $\bM_\Psi$ in $L^2(\bbT; \frakh )$: 
$$
\calU M_\Xi \calU^*=\bM_\Psi, 
\qquad
\Psi(\mu)=\Xi(i\tfrac{1+\mu}{1-\mu}), 
\quad
\mu\in\bbT.
$$
Thus, Theorems~\ref{thm.f8} and \ref{thm.f10}  extend to the SHOs on the circle with piecewise
continuous symbols.

Assumption \eqref{e11}  means that the symbol $\Psi$ must be continuous 
at $\mu=1$. But of course this assumption can be lifted by means of a
rotation. 
Indeed, suppose that $\Psi$ has a jump at $\mu=1$. 
Choose $\alpha$   such that $\Psi$ is continuous at $\mu=e^{i \alpha}$. 
Then all above mentioned spectral results are true for the SHO $\bM_{\widetilde \Psi}$ with the symbol $\widetilde \Psi(\mu)= \Psi(\mu e^{i \alpha})$.
Since the operator
$\bM_{\widetilde \Psi}$ is unitarily equivalent to $\bM_\Psi$ 
through the unitary transformation 
$f(\mu)\mapsto f(\mu e^{i \alpha})$, 
we can reformulate these results in terms of the operator $\bM_\Psi$.
This reasoning yields the following spectral results. 

\begin{theorem}\label{thm.f9}
Let a symbol $\Psi (\mu)\in {\mathfrak S}_\infty(\frakh)$ be a norm-continuous function apart from some jump discontinuities at finitely many points 
$ \mu_1, \ldots, \mu_{L}$. Set
$$
K_{\ell}= \lim_{\varepsilon\to+0}\big(\Psi (\mu_{\ell}e^{+i \varepsilon})-\Psi (\mu_{\ell}e^{-i \varepsilon})\big)
$$
and assume that   
$$
\norm{\Psi(\mu_\ell e^{\pm i\eps})-\Psi(\mu_\ell e^{\pm i0})}
=
O(\abs{\log \eps}^{-\beta_0}), \quad \eps\to+0,
$$
with some exponent $\beta_0>2$ for each $\ell=1,\dots,L$. 
Then the a.c.\   spectrum of the  operator $\bM_\Psi$ in $L^2(\bbT; \frakh)$ consists of the union of the intervals in the r.h.s.\  of \eqref{f18}. The singular continuous spectrum of $\bM_\Psi$ is empty. 
The eigenvalues of $\bM_\Psi$ can accumulate only to $0$ and to the points 
$\pm \tfrac12 s_n(K_\ell)$. 
All eigenvalues of $\bM_\Psi$, distinct from 
$0$ and from $\pm \tfrac12 s_n(K_\ell)$, 
have finite multiplicities.
\end{theorem}

The results of Theorem~\ref{thm.f10} concerning the wave operators can also be extended to SHOs on the unit circle if model operators $M_{\Xi_{\ell}}$ are transplanted into the space $L^2(\bbT; \frakh)$ via the unitary transform \eqref{e20}.

\begin{remark}\label{rmk}
Using the unitary operator $\calU$ (see \eqref{e20}) to transform $\bM_\Psi$ back to $M_\Xi$, we 
see that the assumption \eqref{e11} in Theorem~\ref{thm.f8} 
can be relaxed. Instead, one can assume that the limits
$\Xi(\pm\infty)$ exist and 
$$
\norm{\Xi(\pm\lambda)- \Xi(\pm\infty)}
=
O(\abs{\log \lambda}^{-\beta_0}), 
\quad
\lambda\to+\infty, 
$$
with $\beta_0>2$. Then the jump    $K_\infty=\Xi(-\infty)-\Xi(+\infty)$ at infinity
will also contribute to the orthogonal sum \eqref{f18}. 
\end{remark}

\section{Representations for $\DD_\varphi$}\label{sec.b6}

For an arbitrary bounded function $\varphi$, we define the operator $\DD_\varphi$  by formula \eqref{a0}. Our goal here is to derive a formula (see Theorem~\ref{lma.d2}) for the operator $  \DD_\varphi  $ sandwiched between appropriate functions of $H_0$
in the spectral representation of $H_0$. This formula motivates  and 
explains  much of our construction.  Technically, we need also another representation  (see Theorem~\ref{reprX}) for   $  \DD_\varphi  $ sandwiched between   functions of $H_0$ with disjoint supports.   

We note that our  representations for   $  \DD_\varphi  $  are different from the one given by a double operator integral (see the survey \cite{BS} by M.~Sh.~Birman and M.~Z.~Solomyak and references therein). In particular, the double operator integral approach  treats the operators $H_{0}$ and $H$ in a symmetric way while our representations rely on  the spectral representation of $H_0$ only. This is convenient for our purposes.

\subsection{Two formulas}\label{sec.b61}

Under Assumption~\ref{as1}
we set $\Delta=\Delta_{1}\cup\cdots \cup \Delta_{L}$ and define the unitary map
\begin{equation}
\calF_\Delta : \Ran E_{H_{0}}( \Delta )\to L^2(\Delta_{1};\calN_{1}) \oplus\cdots \oplus L^2(\Delta_L;\calN_L)=:\mathcal{K}
\label{eq:FF}\end{equation}
by the relation $\calF_{\Delta}f=\calF_{\ell}f$ if $f\in \Ran E_{H_{0}}(\Delta_{\ell})$. We extend the
unitary operator $\calF_\ell:\Ran E_{H_0}(\Delta_{\ell}) \to L^2(\Delta_{\ell};\calN_{\ell})$
  by zero to $ \Ran E_{H_0}({\bbR}\setminus\Delta_{\ell})$ and consider it as the map $\calF_{\ell}: \calH\to L^2(\bbR;\calN )$ where
 \begin{equation}
\calN =\calN_{1}\oplus\cdots \oplus \calN_{L}.
\label{eq:NN}\end{equation}
 Then
$\calF_{\Delta} = \calF_{1} + \cdots + \calF_{L}: \calH\to L^2(\bbR;\calN )$ is isometric on  $ \Ran E_{H_{0}}(\Delta )$ and it is zero on $ \Ran E_{H_0}({\bbR}\setminus\Delta )$.
  Note that the adjoint operator $\calF_{\Delta}^* : L^2(\bbR;\calN) \to \calH$ is a partial isometry which sends    the subspace $ \calK$   unitarily onto $ \Ran E_{H_{0}}( \Delta )$ and is zero on the orthogonal complement of $\calK$.

We set $Z(\lambda)=Z_{\ell}(\lambda)$ for $\lambda\in\Delta_{\ell}$ where the operators  $ Z_{\ell}(\lambda)$ are defined in  \eqref{eq:Z}.
 Recall that  the set $\Omega$ is defined by \eqref{eq:Ome} and $Y(z)$ is defined by \eqref{a6a}, \eqref{b7}.  Let $\omega$ and $\varphi$ be bounded functions with compact supports $\supp \omega\subset \Delta$ and $\supp \varphi\subset\Omega$. 
Define the operators
$$
\calZ_{\omega} : L^2(\bbR;\calH)\to L^2( \bbR ;\calN)
\quad \text{ and } \quad
\calY_{\varphi}: L^2(\bbR;\calH)\to L^2(\bbR;\calH)
$$
by formulas
\begin{equation}
(\calZ_{\omega}  u)(\lambda)=
 \omega(\lambda)Z(\lambda)   u(\lambda), 
\quad
(\calY_\varphi u)(\lambda)=
\varphi(\lambda)Y (\lambda+i0) u(\lambda).
\label{d1a}
\end{equation}
Recall that $Z(\lambda)$ and $Y (\lambda+i0)$ are H\"older continuous in $\lambda$ on compact subintervals of $\Delta$ and $\Omega$, respectively. In particular, this implies
that the operators $\calZ_{\omega}$ and $\calY_\varphi$ are bounded.

The following representation of the sanwiched  operator $\DD_\varphi$  is central   to our construction.

\begin{theorem}\label{lma.d2}
Let   Assumption~$\ref{as1}$ be satisfied.
 Let $\omega$ and $\varphi$ be bounded functions with compact supports $\supp \omega\subset \Delta$ and $\supp \varphi\subset\Omega$, and 
let the operators $\calZ_{\omega}$ and $\calY_\varphi$ be defined by  formulas
 \eqref{d1a}. Then   the representation
\begin{equation}
\calF_\Delta \omega(H_{0}) \DD_\varphi \omega(H_{0}) \calF_\Delta^*
=
\calZ_{\omega} ( 
P_- \calY_\varphi  P_+ 
+
P_+ \calY_\varphi^* P_- 
)\calZ_{\omega}^*  
\label{c4a}
\end{equation}
holds.
\end{theorem}

The proof is given in Subsection~\ref{sec.b63}.

Now we consider $\DD_\varphi$ as an operator acting from $\calH$ to the spectral representation of $H_{0}$.  Let $\varphi$ and $  v $   be  bounded functions such that $\supp \varphi\subset\Omega$ and
 \begin{equation} 
  \dist(\supp\varphi,\supp v)>0. 
  \label{eq:supp}
\end{equation}
Then we can define  the operators
$
{\bbY}_{\varphi,v}^{(\pm)}:  \calH\to L^2(\bbR;\calH)
$
by the formula
\begin{equation}
( {\bbY}_{\varphi,v}^{(\pm)} f )(\lambda)=
\varphi(\lambda) Y (\lambda\pm i0) G R_{0}(\lambda) v(H_{0}) f.
\label{d1ay}\end{equation}
In view of condition \eqref{eq:supp} the operators $G R_{0}(\lambda) v(H_{0}) $   are compact in $\calH$ and depend H\"older continuously on $\lambda\in \supp\varphi$.
 
\begin{theorem}\label{reprX}
Let functions $\omega  $ and $\varphi$ be the same as in 
Theorem~$\ref{lma.d2}$. Let $  v $   be a bounded function satisfying condition \eqref{eq:supp}. Then under Assumption~$\ref{as1}$ the representation  
\begin{equation}
2\pi i \calF_\Delta \omega (H_0)  \DD_\varphi v(H_0) = \calZ_{\omega}(P_{-} {\bbY}_{\varphi,v}^{(+)} + P_{+}{\bbY}_{\varphi,v}^{(-)})
\label{eq:repr}
\end{equation}
holds.
\end{theorem}

The proof is given in Subsection~\ref{sec.b64}.

\subsection{Auxiliary results}\label{sec.b62}


We start with a simple identity.  

\begin{lemma}\label{Stone}
Suppose that $\varphi$ is a bounded function and that both the spectra of $H_{0}$ and $H$ are purely a.c.\  on $\supp\varphi$.
Then for   all $f,g\in\calH$, we have the identity
 \begin{align}
(2\pi i)^2 & (\DD_\varphi  f, g)
=
  \int_{-\infty}^\infty    
\lim_{\eps\to+0} (Y(\lambda+i\eps)GR_0(\lambda+i\eps) f,GR_0(\lambda-i\eps) g)
\varphi(\lambda)d\lambda
\nonumber\\
&-   \int_{-\infty}^\infty   \lim_{\eps\to+0} 
(Y(\lambda-i\eps)GR_0(\lambda-i\eps) f,GR_0(\lambda+i\eps) g)\varphi(\lambda)d\lambda
\label{d3}
\end{align}
where the limits in the r.h.s.\  exist  for almost all $\lambda\in\bbR$.
\end{lemma}

\begin{proof}
Since the measure $(E_{H}(\lambda) f,g)$ is absolutely continuous on $\supp\varphi$, it follows from  the spectral theorem  that
$$
(\varphi(H)f,g)
=
  \int_{-\infty}^\infty   \varphi(\lambda)d(E_{H}(\lambda) f,g)=
  \int_{-\infty}^\infty   \varphi(\lambda)\frac{d(E_{H}(\lambda) f,g)}{d\lambda} d\lambda .
$$
Recall that for an arbitrary self-adjoint operator $H$  the relation holds:
$$
2\pi i \frac{d(E_{H}(\lambda) f,g)}{d\lambda}=  
\lim_{\eps\to+0}    \bigl(R(\lambda+i\eps) f,g\bigr) -\lim_{\eps\to+0}    \bigl(R(\lambda-i\eps) f,g\bigr)
$$
where the derivative in the l.h.s.\  and the limits in the r.h.s.\  exist for almost all $\lambda\in\bbR$.  Similar relations are of course also true for the operator $H_{0}$. Putting these relations for $H $ and $H_{0}$ together and collecting terms corresponding to $\lambda+i \varepsilon$ and to $\lambda-i \varepsilon$,  we obtain the formula
\begin{align*}
2\pi i (\DD_\varphi  f, g) =&  \int_{-\infty}^\infty  
  \lim_{\eps\to+0}\bigl((R(\lambda+i\eps)-R_{0}(\lambda+i\eps))f,g\bigr)
   \varphi(\lambda)  d\lambda
\\
&  -\int_{-\infty}^\infty   \lim_{\eps\to+0}\bigl((R(\lambda-i\eps)-R_{0}(\lambda-i\eps))f,g\bigr)   \varphi(\lambda)  
  d\lambda.
\end{align*}
Substituting here   expression \eqref{d2} for $z=\lambda\pm i\varepsilon$,
  we conclude the proof of \eqref{d3}.
  \end{proof}
 
 Our next goal is to pass to the limit $\varepsilon\to 0$ in  \eqref{d3}. 
 This is possible if $\DD_\varphi$ is sandwiched between appropriate functions of $H_{0}$. 
 Passing to the limit relies on the following assertion.

\begin{lemma}\label{Hardy}

For all $f \in\calH$, all bounded functions $\omega$ with compact support $\supp \omega\subset \Delta$ and almost all $\lambda\in\bbR$, we have
\begin{equation}
G  R_0(\lambda\pm i\varepsilon)  \omega(H_{0}) f \to \pm 2\pi i (P_{\pm}\calZ_{\omega}^* \calF_\Delta f)(\lambda)
\label{eq:lim}
\end{equation}
 as $\varepsilon\to+0$.
\end{lemma}

\begin{proof}
It follows  from \eqref{b5} that  
$$
G   R_0(\lambda\pm i\varepsilon) \omega(H_{0}) f
=
\int_{-\infty}^\infty \frac{\omega(x) Z^*(x) (\calF_\Delta E_{H_{0}}(\Delta) f)(x)}{x-\lambda \mp i\varepsilon} dx.
$$
Observe that $\calF_\Delta E_{H_{0}}(\Delta) f\in L^2(\bbR ;\calN)$ and the operator valued function $\omega(x) Z^*(x) $ is bounded. Therefore \eqref{eq:lim} follows from
  formula \eqref{c4}.
 \end{proof}
 
 \subsection{Proof of Theorem~\ref{lma.d2}}\label{sec.b63}
    
Recall (see Proposition~\ref{pr2}) that the spectrum of $H$ is a.c.\  on $\Omega$. Let us  apply   identity \eqref{d3} to the elements $\omega(H_{0})f$ and $\omega(H_{0})g$ instead of $f$ and $g$ and then pass to the limit $\varepsilon\to + 0$.
The operator valued function $Y(\lambda+i\varepsilon)$ converges to  $Y(\lambda+i 0)$ uniformly on $ \supp \varphi$. Therefore it follows from \eqref{d3} and \eqref{eq:lim} that
\begin{multline*}
 (\DD_\varphi \omega(H_{0}) f, \omega(H_{0}) g)
 \\
  = \int_{-\infty}^\infty   
(Y(\lambda+i 0) (P_{+}  \calZ_{\omega}^* \calF_\Delta f)(\lambda)
,(P_{-}  \calZ_{\omega}^* \calF_\Delta g)(\lambda))
\varphi(\lambda)d\lambda
 \\
- \int_{-\infty}^\infty   
(Y(\lambda-i 0) (P_{-}  \calZ_{\omega}^* \calF_\Delta f)(\lambda)
,(P_{+}  \calZ_{\omega}^* \calF_\Delta g)(\lambda))
\varphi(\lambda)d\lambda.
\end{multline*}

Taking   into account \eqref{d4a} and using notation \eqref{d1a}, we get
\begin {align}
 (\DD_\varphi \omega(H_{0}) f, \omega(H_{0}) g)
=
(\calY_\varphi P_+ &\calZ_{\omega} ^*\calF_\Delta f, P_- \calZ _{\omega}^*\calF_\Delta g)_{L^2(\bbR;\calH)}
\nonumber\\
&+
(\calY^*_\varphi P_- \calZ_{\omega} ^*\calF_\Delta f, P_+ \calZ _{\omega}^*\calF_\Delta g)_{L^2(\bbR;\calH)}.
\label{eq:reprV1}\end{align}
Set here $f=\calF_\Delta^* \wt f$,  $g =\calF_\Delta^* \wt g$ where $\wt f$, $\wt g$ are arbitrary elements in $L^2(\bbR;\calH)$ and observe that
 $\calZ_{\omega}=\calF_\Delta \calF_\Delta^*\calZ_{\omega}$.  Thus \eqref{eq:reprV1} implies \eqref{c4a}. 
\qed

\subsection{Proof of Theorem~\ref{reprX}}\label{sec.b64}
 
Let us apply identity \eqref{d3} to the elements $v(H_{0})f$ and $\omega(H_{0})g$ instead of 
$f$ and $g$. We have to pass to the limit $\varepsilon\to + 0$ in the expression
\begin{equation}
(Y(\lambda\pm i\eps)GR_0(\lambda\pm i\eps) v(H_{0}) f,GR_0(\lambda\mp i\eps) \omega(H_{0})g)\varphi(\lambda).
\label{eq:reprV}\end{equation}
 Similarly to the proof of Theorem~\ref{lma.d2}, we use the uniform convergence  of $Y(\lambda+i\eps)$ on $\supp\varphi$  and apply Lemma~\ref{Hardy} to $GR_0(\lambda\pm i\eps) \omega(H_{0})g$. Moreover, we use that 
$$GR_0(\lambda\pm i\eps) v(H_{0}) f\to GR_0(\lambda) v(H_{0}) f$$
as $\varepsilon\to 0$ uniformly on $ \supp \varphi$. Therefore the limit of \eqref{eq:reprV} equals
\begin {multline*}
\pm 2\pi i (Y(\lambda\pm i 0)GR_0(\lambda  ) v(H_{0}) f, (P_{\mp} \calZ^*_{\omega} \calF_\Delta g)(\lambda))\varphi(\lambda)
\\
= \pm 2\pi i  ( ({\bbY}_{\varphi,v}^{(\pm)} f)(\lambda), (P_{\mp} \calZ^*_{\omega} \calF_\Delta g)(\lambda))
 \end{multline*}
 where notation   \eqref{d1ay} has been used. It now follows from \eqref{d3}  that
 \begin {align}
   2\pi i (\DD_\varphi v(H_0) f, \omega (H_0) g)&
\nonumber\\  =
  \int_{-\infty}^\infty ( ({\bbY}_{\varphi,v}^{(+)} f)(\lambda),  (P_{-} &\calZ^*_{\omega} \calF_\Delta g)(\lambda)) d\lambda+ 
    \int_{-\infty}^\infty ( ({\bbY}_{\varphi,v}^{(-)} f)(\lambda), (P_{+} \calZ^*_{\omega} \calF_\Delta g)(\lambda))d\lambda
\nonumber\\  
=( {\bbY}_{\varphi,v}^{(+)} f  , P_{-}& \calZ^*_{\omega} \calF_\Delta g)_{L^2(\bbR;\calH)} + ( {\bbY}_{\varphi,v}^{(-)} f  , P_{+} \calZ^*_{\omega} \calF_\Delta g)_{L^2(\bbR;\calH)}.
\label{eq:reprV2}\end{align}
Set here    $g =\calF_\Delta^* \wt g$ where  $\wt g \in L^2(\bbR;\calH)$ is arbitrary.    Thus \eqref{eq:reprV2} implies \eqref{eq:repr}.  
\qed

\section{Spectral and scattering theory of $\DD_\theta$ }\label{sec.b5}

This section is organized as follows. 
We formulate our main results in Subsection~\ref{sec.b5a} and
give their proofs in Subsection~\ref{sec.d1} modulo an essential analytic result (Theorem~\ref{th.c2}).
The rest of the section is devoted to the proof of Theorem~\ref{th.c2}.

\subsection{Main results }\label{sec.b5a}

Let Assumption~\ref{as1} hold true, and let the operator $\DD_\theta$ be defined by formula \eqref{a0}. 
Our aim is to describe the spectral structure of this operator 
for piecewise continuous functions $\theta$ with 
discontinuities on the set $\Omega$  defined by relation \eqref{eq:Ome}.

\begin{assumption}\label{ass1}
Let $\theta (\lambda)$ be a real function such that:
\begin{itemize}
\item
$\theta$ is continuous apart from   jump discontinuities at the points
$ \lambda_\ell\in \Omega_{\ell}=\Omega\cap\Delta_{\ell}$, $ \ell=1,\ldots, L$;
\item
at each point of discontinuity $\lambda_\ell$, the function $\theta$
satisfies the logarithmic regularity condition
\begin{equation}
\theta(\lambda_\ell\pm\eps)-\theta(\lambda_\ell\pm0)
=
O(\abs{\log\eps}^{-\beta_0}), 
\quad 
\eps\to+0,
\label{b8a}
\end{equation}
with an exponent $\beta_0>2$;
\item
the limits
$\lim_{\lambda\to\pm\infty} \theta(\lambda)$
exist and are finite. 
\end{itemize}
\end{assumption}

Thus we suppose that there is exactly one discontinuity point on every set $\Delta_{\ell}$. To put it differently, given discontinuity points $\lambda_{1},\ldots, \lambda_{L}$, we suppose that the Assumption~\ref{as1}(C) is satisfied only in some neighbourhoods of these points. We also suppose that discontinuity points do not coincide with the eigenvalues of the operator $H$.

Recall that the unitary  operators $\calF_{\ell}$ and $\calF_\Delta$ were defined by formulas  \eqref{eq:Zf} and \eqref{eq:FF}, respectively. As   in the previous section, we extend them to partially isometric operators acting from $\calH$ to $L^2(\bbR;\calN)$ where the space $\calN$ is defined by equality \eqref{eq:NN}.
  We denote the jump of $\theta$ at $\lambda_{\ell}$ by $\varkappa_\ell$, see \eqref{b8b}; 
  $\sigma_{n}(\lambda_{\ell})$, 
$1\leq n\leq N_{\ell}$,  are the eigenvalues of the scattering matrix $S (\lambda_{\ell})$ for the pair of operators $H_{0}$, $H$ and the numbers $a_{n\ell}$ are defined in \eqref{eq:a0}.  
    As before,  the function $q(\lambda)$ is defined by equalities \eqref{c7} and \eqref{e4},  and we denote by the same letter $q$ the operator of multiplication by this function in the space $L^2(\bbR ;\calN )$.

 
  The spectral properties of the operator $\DD_\theta$ are described in the following assertion.

\begin{theorem}\label{th.b1}
Let Assumptions~$\ref{as1}$ and $\ref{ass1}$ hold true. 
Then:
\begin{enumerate}[\rm (i)]

\item
The a.c.\  spectrum of $\DD_\theta$ consists of the union of the intervals $[-a_{n\ell}, a_{n\ell}]$, that is, relation \eqref{eq:a0} holds.

\item
The singular continuous spectrum of $\DD_\theta$ is empty.

\item
The eigenvalues of $\DD_\theta$, distinct from $0$ and from $\pm a_{n\ell}$,
have finite multiplicities and can accumulate only to $0$ and to the 
points $\pm a_{n\ell}$. 

\item
 For any $\beta > 1 $, the operator-valued function 
$q^\beta\calF_\Delta  (\DD_\theta-zI)^{-1}\calF_\Delta ^*q^\beta$ is   continuous
in $z$ for $\pm\Im z\geq 0$   away    from  $0$, all points  $\pm a_{n\ell}$ and the eigenvalues of $\DD_\theta$. 
\end{enumerate}
\end{theorem}
 
 To construct scattering theory for the operator $ \DD_\theta$, we have to introduce a model operator for each discontinuity $\lambda_\ell$ of the function $\theta$.   Let $\zeta$ be as in \eqref{c6}.  For $\ell=1,\dots,L$, we define the operator $K_\ell  : \calN \to \calN $ by the equalities
\begin{equation}
 K_\ell \psi=\varkappa_\ell    (S(\lambda_\ell)-I)\psi\quad {\rm if}\quad \psi\in\calN_{\ell}\quad {\rm and}\quad K_\ell \psi=0 \quad {\rm if}\quad\psi\in  \calN_{\ell}^\bot
 \label{c7ax}\end{equation}
and then define the symbol
$\Xi_\ell(\lambda)$ by formula \eqref{e9}.
We note that the function  $\Xi_\ell$ is bounded and 
has a single point of discontinuity 
at $\lambda=\lambda_\ell$. 
It is clear that the singular values of $S(\lambda_\ell)-I$ are 
$ \abs{\sigma_n(\lambda_\ell)-1}$. According to Lemma~\ref{Me1}, each operator $M_{\Xi_\ell}$ 
can be explicitly diagonalized.  Except for a possible zero eigenvalue   of infinite multiplicity, 
its spectrum is absolutely continuous and consists of the intervals $[-a_{n\ell}, a_{n\ell}]$, 
$1\leq n\leq N_{\ell}$. 
Note that $\Ker K_{\ell}\neq \{0\}$ and hence
$\Ker M_{\Xi_\ell}\neq \{0\}$  if and only if $a_{n\ell}=0$ for some $n$. 
Even for $L=1$, this happens if the scattering matrix $S(\lambda_{1})$ has the eigenvalue $1$. We emphasize that the zero eigenvalue   of the operators $M_{\Xi_\ell}$ is irrelevant for our construction.


\begin{theorem}\label{th.c3}
Let Assumptions~$\ref{as1}$ and $\ref{ass1}$ hold true.
Then:
\begin{enumerate}[\rm (i)]
\item
The wave operators
\begin{equation}
W_\pm(\DD_\theta, M_{\Xi_\ell})
:=
\slim_{t\to\pm\infty} e^{i\DD_\theta t}\calF_{\ell}^* e^{-i M_{\Xi_\ell} t} P^{({\rm ac} )}_{M_{\Xi_\ell} },
\quad 
\ell=1,\dots,L,
\label{b13}
\end{equation}
exist and enjoy the intertwining property
$$
\DD_\theta W_\pm(\DD_\theta, M_{\Xi_\ell})=W_\pm(\DD_\theta,M_{\Xi_\ell})M_{\Xi_\ell}.
$$
The wave operators are isometric on the a.c.\  subspaces of  $M_{\Xi_\ell}$  and their ranges are orthogonal to each 
other, i.e. 
\begin{equation}
\Ran W_\pm(\DD_\theta, M_{\Xi_j})\perp \Ran W_\pm(\DD_\theta,M_{\Xi_k}), 
\quad j \neq k.
\label{b14}
\end{equation}
\item
The asymptotic completeness holds: 
\begin{equation}
\Ran W_\pm(\DD_\theta,M_{\Xi_1})\oplus\dots\oplus \Ran W_\pm(\DD_\theta, M_{\Xi_L})
=
\calH_{\DD_\theta}^{({\rm ac})} .
\label{b15}
\end{equation}
  \end{enumerate} 
\end{theorem}

\begin{corollary}\label{cor.c3}
 
 For an arbitrary $f\in\calH_{\DD_\theta}^{({\rm ac})}$, we have the relation
$$
\lim_{t\to\pm \infty} \| e^{-i\DD_\theta t}f -\sum_{\ell=1}^L e^{-i M_{\Xi_\ell} t} f_{\ell}\|=0, \quad f_{\ell}= W_\pm^*(\DD_\theta, M_{\Xi_\ell}) f.
$$
  \end{corollary}
  
The interval $\Delta_{\ell}$ in the definition of the operator $\calF_{\ell}$ can be replaced by a smaller one. In view of Lemma~\ref{as} this  does not change   the wave operator $W_\pm(\DD_\theta, M_{\Xi_\ell})$. 	   We note also that this operator is not changed if the model operator $M_{\Xi_\ell}$ is considered in the space $L^2(\bbR; \calN_{\ell})$ and 
$\calF_{\ell}$ is considered as the mapping $\calF_{\ell}:\calH \to L^2(\bbR; \calN_{\ell})$.

\subsection{Proofs of Theorems~\ref{th.b1} and  \ref{th.c3}}\label{sec.d1}

Our proofs of Theorems~\ref{th.b1} and  \ref{th.c3} rely on Propositions~\ref{pr.c1} and  \ref{pr.c2}, respectively. They  bear a certain resemblance to the proofs of 
Theorems~\ref{thm.f8} and \ref{thm.f10}. 
Indeed, the symbols of the model operators are defined
by the same formula \eqref{e9}. Now the auxiliary space   $\frakh=\calN$ and the operator $K_{\ell}$ is given by \eqref{c7ax}.  On the contrary, the role of the operator $A$ in Proposition~\ref{pr.c1}  will now be played by the operator
$ \DD_\theta$ transplanted  by an isometric transformation into the space $ L^2(\bbR;\calN)$ whereas in Section~\ref{sec.ee} the operator $A$ was itself a SHO. This difference leads to rather serious technical difficulties.

   Let us consider an arbitrary isometric transformation
  \begin{equation}
\calF_\bot : \Ran E_{H_0}( {\bbR}\setminus\Delta)\to L^2({\bbR}\setminus\Delta;\calN).
\label{eq:bot}\end{equation}
 Without loss of generality we may assume that $| {\bbR}\setminus\Delta|>0$ so that such a mapping exists. Then
the operator $\calF=\calF_\Delta \oplus\calF_\bot:\calH\to  L^2({\bbR} ;\calN)$ is also isometric.  Of course, the construction of the operators $\calF_\bot$ and hence of $\calF $ is   not unique and has a slightly artificial flavour. However, it is very convenient because it allows us to develop the scattering theory in only one space. Note that if $ E_{H_0}( \Delta)= I$ (this is the case, for example, for the Schr\"odinger operator -- Example~\ref{Schr}), then $\calF =\calF_\Delta$.

 The following result shows that the operator  $\DD_\theta$ transplanted by our isometric (but, in general, not unitary) transformation $\calF $ into the space $L^2(\bbR;\calN)$ is well approximated by the sum of model operators.

\begin{theorem}\label{th.c2}
Let Assumptions~$\ref{as1}$ and $\ref{ass1}$  hold true, and let the function $q$ be defined by formulas \eqref{c7} and \eqref{e4}.  Set  
\begin{equation}
A_\infty = \calF  \DD_\theta \calF ^* -\sum_{\ell=1}^L M_{\Xi_\ell}.
\label{d27}\end{equation}
Then the operator $q^{-\beta} A_\infty  q^{-\beta}$ is compact 
in $L^2(\bbR;\calN)$ for all  $\beta<\beta_0/2$. 
\end{theorem}

The proof of Theorem~\ref{th.c2} is lengthy and will be given in the following subsections.

 Given Theorem~\ref{th.c2}, the proofs of Theorems~\ref{th.b1} and  \ref{th.c3} are almost identical to those of Theorems~\ref{thm.f8} and \ref{thm.f10}.
We   check that the assumptions of  Proposition~\ref{pr.c1} are satisfied for the operators $A=\calF  \DD_\theta \calF ^*$ and $A_\ell=M_{\Xi_\ell}$, 
$\ell=1,\dots,L$, acting in the space $ L^2(\bbR;\calN)$. Equality \eqref{c3} is now true with the operator $A_\infty$ defined by \eqref{d27}.  We take $X=q^\beta$  where   $\beta$ satisfies $1<\beta<\beta_0/2$, $\beta<3/2$. 
Let  $\delta $ be any bounded  open interval which does not contain 
the points  $0$, $\pm a_{n\ell}$,   $\ell=1,\dots, L$,  $1\leq n\leq N_{\ell}$.  The symbol of the SHO $M_{\Xi_\ell}$   is a particular case of \eqref{e9} corresponding to the operator $K_{\ell}$ defined by formula \eqref{c7ax}.
  Since $ a_{n\ell}=2^{-1} s_{n} (K_{\ell})$, the choice of 
 $\delta $ is also the same as in Subsection~\ref{sec.e6}. Thus all the assumptions of  Proposition~\ref{pr.c1}, except \eqref{c2}, have been already verified there. Finally, condition \eqref{c2}  follows from Theorem~\ref{th.c2}. Thus all assertions of Propositions~\ref{pr.c1} and \ref{pr.c2} are true for the  operators $ \calF  \DD_\theta \calF ^*$ and $ M_{\Xi_1},\ldots, M_{\Xi_{L}}$.
 
 It remains to reformulate the results in terms of the  operator $  \DD_\theta $. Since $\calF ^*\calF =I$,  the restriction of $ \calF  \DD_\theta \calF ^*$ onto  $\Ran  \calF$ is unitarily equivalent to the  operator $  \DD_\theta $ and $ \calF  \DD_\theta \calF ^* f =0$  for $f$ in the orthogonal complement to $\Ran  \calF$. It follows that
  \begin{equation}
 \calF ^*  \varphi(\calF  \DD_\theta \calF ^*)  = \varphi(\DD_\theta )  \calF ^*
\label{eq:FU}\end{equation}
for, say, continuous functions $\varphi$.
In particular, the spectra of  $  \DD_\theta $ and $ \calF  \DD_\theta \calF ^*$ coincide up to a possible zero eigenvalue. This yields
 the first three statements of Theorem~\ref{th.b1}. According to \eqref{eq:FU} for $\varphi(\lambda)= (\lambda-z)^{-1}$, we have
   \begin{align*}
  q^\beta\calF_{\Delta} (\DD_\theta-zI)^{-1}\calF_{\Delta}^*  q^\beta  =&  q^\beta\calF_{\Delta}\big(\calF^*   (\calF \DD_\theta \calF^* - zI)^{-1}  \calF\big)\calF_{\Delta}^*  q^\beta 
\\
  =&  \big(\calF_{\Delta}\calF^*\big)  \big( q^\beta (\calF \DD_\theta \calF^* - zI)^{-1} q^\beta \big) \big(\calF\calF_{\Delta}^*  \big)  . 
 \end{align*}
 At the last step we have used that the operator  $q^\beta$ commutes with
 $\calF_\Delta\calF^* $ which is the orthogonal projection in $ L^2(\bbR;\calN)$ onto the subspace $\mathcal{K}$ defined in \eqref{eq:FF}. Therefore
the last statement of Theorem~\ref{th.b1} follows from the continuity of the operator valued function $  q^\beta (\calF \DD_\theta \calF^* - zI)^{-1} q^\beta$.

Proposition~\ref{pr.c2} gives the existence of the wave operators $W_{\pm}( \calF  \DD_\theta \calF ^*,  M_{\Xi_{\ell}})$   so that according to \eqref{eq:FU} for $\varphi(\lambda)=e^{i\lambda t}$ there exist
\begin{align}
\slim_{t\to\pm\infty} e^{i\DD_\theta t}\calF^*   e^{-i M_{\Xi_\ell} t} P^{(\ac)}_{M_{\Xi_\ell}} &=
\slim_{t\to\pm\infty}\calF^*  e^{i \calF  \DD_\theta \calF ^* t}    e^{-i M_{\Xi_\ell} t} P^{(\ac)}_{M_{\Xi_\ell}} 
\nonumber\\
&=  \calF^* W_{\pm}( \calF  \DD_\theta \calF ^*,  M_{\Xi_{\ell}}) 
\label{b13n}\end{align}
for all $\ell=1,\ldots,L$.
Since $\chi_{\bbR\setminus\Delta}\calF_\bot=\calF_\bot$ and $\chi_{ \Delta_{k}}\calF_k=\calF_k$, it follows from Lemma~\ref{as} that
$$
\slim_{t\to\pm\infty} \calF_\bot ^*   e^{-i M_{\Xi_\ell} t} P^{(\ac)}_{M_{\Xi_\ell}}=0 \quad {\rm  and} \quad\slim_{t\to\pm\infty} \calF_k ^*   e^{-i M_{\Xi_\ell} t} P^{(\ac)}_{M_{\Xi_\ell}}=0, \quad k\neq \ell.
$$
Therefore relation \eqref{b13n} implies that the limits  \eqref{b13}  exist and
\begin{equation}
 W_{\pm}(  \DD_\theta ,  M_{\Xi_{\ell}}) =  \calF^* W_{\pm}( \calF  \DD_\theta \calF ^*,  M_{\Xi_{\ell}}), \quad \ell=1,\ldots,L.
\label{eq:asc}\end{equation}
The  intertwining property of wave operators is a direct consequence of their existence.
Since $\calF_{\ell} \calF_{\ell}^* P^{(\ac)}_{M_{\Xi_\ell}}=\chi_{ \Delta_{\ell}}P^{(\ac)}_{M_{\Xi_\ell}}$,
Lemma~\ref{as} implies that
$$
\slim_{t\to\pm \infty}( \calF_{\ell} \calF_{\ell}^* -I ) e^{-i M_{\Xi_\ell} t} P^{(\ac)}_{M_{\Xi_\ell}} =0, 
$$
and hence the wave operators $W_{\pm}( \calF  \DD_\theta \calF ^*,  M_{\Xi_{\ell}})$ are isometric. Relations \eqref{b14} are trivial because $\calF_{\ell}=\calF_{\ell} E_{H_{0}}(\Delta_{\ell})$ and hence
$$
\calF_{k}\calF_{j}^*=\calF_{k} E_{H_{0}}(\Delta_{k} \cap \Delta_j )\calF_{j}^*=0, \quad j\neq k.
$$

Finally, according to \eqref{b12} we have
\begin{equation}
\Ran W_\pm(\calF \DD_\theta \calF^*,M_{\Xi_1})\oplus\dots\oplus \Ran W_\pm(\calF \DD_\theta \calF^*, M_{\Xi_L})
=
\calH_{\calF \DD_\theta \calF^*}^{(\ac)}.  
\label{eq:asc1}\end{equation}
 This relation implies \eqref{b15}. 
  Indeed, observe that $\calH_{  \DD_\theta  }^{(\ac)}=\calF^* \calH_{\calF \DD_\theta \calF^*}^{(\ac)}  $. Hence
if $f\in \calH_{  \DD_\theta  }^{(\ac)}$, then $f=\calF^* \wt{f}$ where $\wt{f}\in \calH_{\calF \DD_\theta \calF^*}^{(\ac)} $.  It follows from \eqref{eq:asc1} that
$$
\wt{f}=\sum_{\ell=1}^L W_\pm(\calF \DD_\theta \calF^*,M_{\Xi_\ell})\wt{f}_{\ell} \quad {\rm where} \quad \wt{f}_{\ell}=W_\pm^*(\calF \DD_\theta \calF^*,M_{\Xi_\ell}) \wt{f}.
$$
Therefore 
$$ 
 f=  \sum_{\ell=1}^L \calF^* W_\pm(\calF \DD_\theta \calF^*,M_{\Xi_\ell})\wt{f}_{\ell}=
\sum_{\ell=1}^L   W_\pm(  \DD_\theta ,M_{\Xi_\ell})\wt{f}_{\ell} 
$$
according to  equality \eqref{eq:asc}. 
This proves \eqref{b15} and hence concludes the proof of Theorem~\ref{th.c3}. \qed

Just as with the analysis of SHOs (see Remark~\ref{rmk1}) in the case $L=1$ many of the above steps simplify considerably.

\subsection{Proof of Theorem~\ref{th.c2}}\label{sec.d1w}

Here we briefly describe the  main steps of the proof of Theorem~\ref{th.c2}. We consecutively replace the operator $ \calF  \DD_\theta \calF ^*$ by simpler operators 
neglecting all terms that admit the representation $q^{\beta} K q^{\beta}$ with an 
operator $K$ compact in the space $L^2(\bbR;\calN)$. We call such terms {\it negligible}. Our goal is to reduce the operator $ \calF  \DD_\theta \calF ^*$ to the SHO with the explicit symbol
\begin{equation}
\Xi_{0} (\lambda)=\sum_{\ell=1}^L \Xi_\ell (\lambda).
\label{eq:symb}\end{equation}

The first step    is   to replace the function $\theta$ by a function $ \varphi$ supported in $\Omega$. To be more precise, we choose a function $\rho\in C_0^\infty(\Omega)$   such that 
$\rho (\lambda)=1$ in an open neighbourhood of 
$\{\lambda_1,\dots,\lambda_L\}$ and set
$$
\varphi (\lambda)= \rho(\lambda)\theta(\lambda).
$$
The operator $ \calF ( \DD_\theta -\DD_\varphi) \calF ^*$ is negligible in view of the following

\begin{lemma}\label{lma.d4}
Let $\xi\in C(\bbR)$ be such that the limits $\lim_{\lambda\to\pm\infty}\xi(\lambda)$ 
exist and are finite and assume that $\xi (\lambda)=0$ in a
neighbourhood of $\{\lambda_1,\dots,\lambda_L\}$.
Then for any $\beta>0$,
$$
q^{-\beta}\calF \DD_\xi \calF^* q^{-\beta}\in\mathfrak S_\infty.
$$
\end{lemma}

The second step   is to sandwich  $\DD_\varphi$ between the operators $\omega(H_{0})$. We suppose that   $\omega\in C_{0}^\infty(\Delta)$ and $\omega(\lambda)=1$ for $\lambda\in \supp\varphi$. 

\begin{lemma}\label{lma.d4m}
For any $\beta>0$, the difference
$$
q^{-\beta}\calF \big(\omega(H_0)\DD_\varphi \omega(H_0)-\DD_\varphi\big)\calF^*q^{-\beta}
$$
is compact. 
\end{lemma}

Observe that $\calF \omega(H_0)=\calF_\Delta\omega(H_0)$. Thus, up to negligible terms,   we get the   operator $\calF_\Delta\omega(H_0)\DD_\varphi \omega(H_0) \calF_\Delta^*$ which localizes the problem onto neighbourhoods of singular points. It is important that the   operator $\calF_\Delta\omega(H_0)\DD_\varphi \omega(H_0) \calF_\Delta^*$ admits representation \eqref{c4a}. 

Next, we reduce the problem to the study of SHO. To that end, we swap   the operators $\calZ_{\omega}$ and $P_{\pm}$ in the r.h.s.\  of \eqref{c4a}. In fact, we have

\begin{lemma}\label{lma.d4n}
Let  
\begin{equation}
\Xi(\lambda)
= \omega^2(\lambda) \varphi(\lambda)
 Z (\lambda) Y(\lambda+i0) Z^* (\lambda) :\calN\to\calN .  
\label{eq:xi}\end{equation} 
Then,
for any $\beta>0$, the difference
$$
q^{-\beta}\big(\calF_\Delta\omega(H_0)\DD_\varphi \omega(H_0) \calF_\Delta^*-M_\Xi \big)q^{-\beta}
$$
is compact. 
\end{lemma}

Thus the problem reduces to the analysis of the SHO $M_\Xi$. Up to negligible terms, it is determined by the values $\Xi (\lambda_{\ell})$ at the points of discontinuity of $\theta(\lambda)$. Putting together    the stationary representation 
\eqref{b8} of the scattering matrix and the definition \eqref{e9} of $\Xi_\ell$, we will prove the following result.

\begin{lemma}\label{lma.d4k}
Let $\Xi_{0}$ be given by formula \eqref{eq:symb}.
Then the symbol
$\Xi- \Xi_{0}$
satisfies the hypothesis of Lemma~$\ref{lma.d1}$.
\end{lemma}

Combining this result with Corollary~\ref{lma.e2c}, we directly obtain

\begin{lemma}\label{lma.d4ks}
For any $\beta>0$, the difference
$$
q^{-\beta}  (M_\Xi - M_{\Xi_{0}} )q^{-\beta}
$$
is compact. 
\end{lemma}

Theorem~\ref{th.c2} follows from Lemmas~\ref{lma.d4}, \ref{lma.d4m}, \ref{lma.d4n} and \ref{lma.d4ks}. Thus for the proof of Theorem~\ref{th.c2} it remains to establish Lemmas~\ref{lma.d4}, \ref{lma.d4m}, \ref{lma.d4n} and \ref{lma.d4k}. This requires several analytic assertions which are collected in the next subsection.

\subsection{Compactness of the sandwiched operators $\DD_{\varphi}$}\label{sec.d2}

For the proof of the first assertion, see \cite[Theorem~7.3]{Push3} 
and \cite[Lemma~5.4]{PY1}.
\begin{lemma}\label{lma.d3a}
Under Assumption~$\ref{as1}(B)$  the operator $\DD_\xi$ 
is compact for any function $\xi\in C(\bbR)$ 
such that the limits $\lim_{\lambda\to\pm\infty}\xi (\lambda)$
exist and are finite. 
\end{lemma}

The next one is  a direct consequence of our construction of the operator $\calF$.

\begin{lemma}\label{F}
Suppose that $v\in L^\infty (\bbR)$ and $vq^{-\beta}\in L^\infty (\bbR)$. Then the operator $ q^{-\beta}\calF v(H_{0}):{\calH}\to L^2 (\bbR; \calN)$ is bounded.
\end{lemma}

\begin{proof}
The operator $ q^{-\beta}\calF_\Delta v(H_{0})=q^{-\beta}v\calF_\Delta  $ is bounded because
$vq^{-\beta}\in L^\infty (\bbR)$. The operator $ q^{-\beta} \calF_\bot v(H_{0})  $ is bounded because, by \eqref{eq:bot}, $\calF_\bot=\chi_{\bbR\setminus \Delta}\calF_\bot$ and 
$q^{-\beta} \chi_{\bbR\setminus \Delta}\in L^\infty(\bbR)$.
\end{proof}

The following assertion relies on Theorem~\ref{reprX}.

 
\begin{lemma}\label{lma.d2a}
Under the assumptions of Theorem~$\ref{reprX}$ for any $\beta>0$, the operator
$$
q^{-\beta} \calF\omega (H_0) \DD_\varphi v(H_0)
: \calH \to L^2(\bbR; \calN)
$$
is compact.
\end{lemma}
 
\begin{proof}
Since $\calF  \omega (H_0)=\calF_\Delta \omega (H_0)$, we can use representation \eqref{eq:repr} where   $ {\bbY}_{\varphi,v}^{(\pm)} $ is the operator \eqref{d1ay}. Thus we have to check that the operators
$$
q^{-\beta}\calZ_{\omega} P_{\pm} {\bbY}_{\varphi,v}^{(\mp)} :
\calH\to L^2 (\bbR; \calN)
$$
are compact. 
Observe that $q^{-\beta}\calZ_\omega =\calZ_\omega q^{-\beta}$, the operator
$\calZ_\omega: L^2 (\bbR; \calH)\to L^2 (\bbR; \calN)$ is bounded  and, by Proposition~\ref{Muck},  the operator $q^{-\beta}  P_{\pm} q^{\beta}$ is  bounded in $L^2 (\bbR; \calH)$. 
Therefore it suffices to verify the compactness of the operator
$$
q^{-\beta}{\bbY}_{\varphi,v}^{(\mp)}: \calH\to L^2 (\bbR; \calH).
$$
Set 
\begin{equation}
   Y_{v}^{(\pm)}  (\lambda)=
 Y (\lambda\pm i0) G R_{0}(\lambda) v(H_{0}),  \quad \lambda\in \supp \varphi.
\label{d1ax}\end{equation}
It follows from definition \eqref{d1ay}  that
$$
(q^{-\beta}{\bbY}_{\varphi ,v}^{(\mp)}f)(\lambda)=q^{-\beta}(\lambda)\varphi(\lambda)
Y_{ v}^{(\mp)}(\lambda) f.
$$
Suppose that $f_{n}\to 0$ weakly in $\calH$ as $n\to\infty$. 
By definition \eqref{d1ax}, the operators $Y_{ v}^{(\mp)}(\lambda) $ are compact, and hence
$\|Y_{ v}^{(\mp)}(\lambda) f_{n}\|\to 0$ as $n\to\infty$ for $\lambda\in\supp\varphi$. 
Note that the integrand in 
\begin{equation}
 \| q^{-\beta}{\bbY}_{\varphi,v}^{(\mp)}f_{n}\|^2
 =\int_{-\infty}^\infty q^{-2\beta}(\lambda)\varphi^2(\lambda)
\|Y_{ v}^{(\mp)}(\lambda) f_{n}\|^2
d\lambda 
\label{eq:S11w}\end{equation}
 is bounded by $Cq^{-2\beta}(\lambda)\varphi^2(\lambda)
\|Y_{ v}^{(\mp)}(\lambda)  \|^2$
which belongs to $L^1({\bbR})$ because the operators $Y_{ v}^{(\mp)}(\lambda) $ are uniformly bounded on $\supp\varphi$.
  Thus expression \eqref{eq:S11w} tends to zero as $n\to\infty$ by the Lebesgue theorem.
  \end{proof}  
 
 Finally, we use Theorem~\ref{lma.d2}.

\begin{lemma}\label{lma.d3}

Let $\rho\in C_0^\infty(\Omega)$; 
then for any $\beta>0$, we have 
$$
q^{-\beta}\calF \DD_\rho  \in \mathfrak S_\infty.
$$
\end{lemma}

\begin{proof}
Let $\omega$ be a bounded function with $\supp\omega\subset \Delta$ such that $\omega(\lambda)=1$ for $\lambda \in\supp \rho$; set $\widetilde\omega=1-\omega$.
    The operator $\DD_\rho$ is compact by Lemma~\ref{lma.d3a}. Therefore $q^{-\beta}\calF \widetilde\omega (H_{0}) \DD_\rho \in \mathfrak S_\infty$  because the operator $q^{-\beta}\calF \widetilde\omega (H_{0})$ is bounded by Lemma~\ref{F}.  It remains to check that $q^{-\beta}\calF  \omega (H_{0}) \DD_\rho \in \mathfrak S_\infty$.
Since $\calF  \omega(H_{0}) =\calF_\Delta \omega(H_{0})$ and $   \omega(H_{0})=  \omega(H_{0})\calF_\Delta^* \calF_\Delta$, to that end we have to verify two inclusions:
\begin{equation}
q^{-\beta}\calF_\Delta \omega(H_{0}) \DD_\rho \omega(H_{0}) \calF_\Delta^*\in \mathfrak S_\infty\quad {\rm and}\quad
q^{-\beta}\calF_\Delta  \omega(H_{0})  \DD_\rho \widetilde\omega (H_{0}) \in \mathfrak S_\infty.
\label{d5x}
\end{equation} 
According to Theorem~\ref{lma.d2}, the first operator here equals
$$
 \calZ_{\omega} q^{-\beta}
(P_-\calY_{\rho} P_+ + P_+\calY^*_{\rho} P_-)\calZ^*_{\omega}
$$
where the operators $ \calZ_{\omega} $ and $\calY_{\rho}$ are defined by formulas \eqref{d1a}. Since the operator $\calZ_{\omega} : L^2(\bbR;\calH)\to L^2( \bbR ;\calN)$ is bounded, it suffices to show that the operator
$q^{-\beta}(P_- Y \rho P_+ + P_+ Y^* \rho  P_-)$
is compact in the space $L^2(\bbR ; \calH)$. This fact   
follows from Corollary~\ref{lma.e2c} applied to  $\Xi=Y\rho$ because the operator valued function $Y(\lambda)$ takes compact values and is H\"older continuous for $\lambda\in\supp\rho$. The second operator in  \eqref{d5x} is compact
according to Lemma~\ref{lma.d2a} where $\varphi=\rho$ and $v=\widetilde \omega$. 
\end{proof}

\subsection{Proof of Lemma~\ref{lma.d4}}\label{sec.d2x}
  
  Let $\rho\in C_0^\infty(\Omega)$ be such that 
$\rho(\lambda)=1$ in a neighbourhood of 
$\{\lambda_1,\dots,\lambda_L\}$ and $\rho(\lambda)\xi(\lambda)=0$. We set
$\wt\rho (\lambda)=1-\rho (\lambda)$ so that $\wt\rho(\lambda)\xi(\lambda)=\xi(\lambda)$. Then we have
\begin{align*}
\xi(H)=&\rho(H_{0}) \xi(H_{0}) + \wt\rho(H) \xi(H) \wt\rho(H)+\wt\rho(H)  \xi(H_{0})\rho(H_{0}),
\\
\xi(H_{0})=&\rho(H ) \xi(H_{0}) + \wt\rho(H) \xi(H_{0}) \wt\rho(H)+\wt\rho(H)  \xi(H_{0})\rho(H)
\end{align*}
and hence
\begin{equation}
\DD_{\xi}
=
-\DD_{\rho} \xi(H_0)
+
\wt\rho(H) \DD_{\xi }\wt\rho(H)
-
\wt\rho(H)\xi(H_0)\DD_{\rho}.
\label{eq:XX1w}\end{equation}
Let us sandwich this expression by $q^{-\beta}  \calF $ and consider every term in the r.h.s.\  separately.

The first term in \eqref{eq:XX1w} yields
$$
- \big(  q^{-\beta}\calF \DD_{\rho}\big) \big( \xi(H_0)\calF^*  q^{-\beta}\big).
$$
 The first factor here is a compact operator by Lemma~\ref{lma.d3} and   the second factor is a bounded operator by   Lemma~\ref{F}.

The second term in \eqref{eq:XX1w} yields
\begin{equation}
\big(  q^{-\beta}\calF \wt\rho(H)\big) \DD_{\xi}  \big( \wt\rho(H)\calF^*  q^{-\beta}\big).
\label{eq:XX}\end{equation}
   We have
 \begin{equation}
  q^{-\beta}\calF \wt\rho(H)= -q^{-\beta}\calF \DD_{\rho}+ q^{-\beta}\calF \wt\rho(H_{0})\in\calB.
\label{eq:XX1}\end{equation}
  Indeed, the first operator on the right is compact according to Lemma~\ref{lma.d3},  and the second operator is bounded according to   Lemma~\ref{F}. Thus the first and third factors in \eqref{eq:XX} are bounded operators.  
  It remains to use the fact that the operator $\DD_{\xi}$ is compact by Lemma~\ref{lma.d3a}. 
 
 Finally, the third term in \eqref{eq:XX1w} yields
$$
- \big(  q^{-\beta}\calF\wt\rho(H) \big) \xi (H_{0}) \big( \DD_{\rho}  \calF^*  q^{-\beta}\big).
$$
 The first factor here is bounded according to \eqref{eq:XX1}, and the last  factor  is compact according to Lemma~\ref{lma.d3}. \qed

\subsection{Proof of Lemma~\ref{lma.d4m}}\label{sec.d2k}
  
Set $\wt\omega=1-\omega$. We have to check two inclusions
\begin{equation}
q^{-\beta}\calF \omega(H_0) \DD_\varphi \wt \omega(H_0)\calF^* q^{-\beta}
 \in\mathfrak S_\infty
\label{d15bb}
\end{equation}
and
\begin{equation}
q^{-\beta}\calF \wt \omega(H_0) \DD_\varphi \wt \omega(H_0)\calF^* q^{-\beta}
 \in\mathfrak S_\infty .
\label{d15b}
\end{equation}

Let $v$ be a $C^\infty$ function satisfying condition \eqref{eq:supp} and such that $v\wt\omega=\wt\omega$. We write operator \eqref{d15bb} as a product of two factors
$$
 \big(q^{-\beta}\calF \omega(H_0) \DD_\varphi v(H_0)\big)\,\big(\wt \omega(H_0)\calF^* q^{-\beta}\big).
$$
 The first one is compact according to Lemma~\ref{lma.d2a}, and the second one is bounded according to Lemma~\ref{F}.
 
  Operator \eqref{d15b} can be factorized into a product of three factors
\begin{equation}
\big(q^{-\beta}\calF v(H_{0}) \big) \,\big( \wt \omega(H_0) \DD_\varphi \big)\, \big(\wt \omega(H_0)\calF^* q^{-\beta}\big).
\label{eq:f}\end{equation}
 The first and the third factors here are bounded operators by Lemma~\ref{F}. Since $\wt\omega (\lambda) \varphi(\lambda)=0$, we can write the second factor as
 \begin{equation}
\wt\omega(H_0)\DD_{\varphi} =\wt\omega(H_0)\varphi(H)
= -\DD_{\wt\omega}\varphi(H)=
\DD_{\omega}\varphi(H).
\label{eq:id}\end{equation} 
 By Lemma~\ref{lma.d3a} the operators $\DD_{\omega}$ and hence \eqref{eq:id} are compact. This proves that operator \eqref{eq:f} is also compact.
\qed 
  
\subsection{Proof of Lemma~\ref{lma.d4n}}\label{sec.d2d}

According to representation \eqref{c4a} and the definition \eqref{b9a} of the SHO $M_\Xi$, we have to check that the operator
\begin{multline*}
q^{-\beta} \calZ_\omega P_- \calY_{\varphi} P_+ \calZ_\omega^* q^{-\beta}
-
q^{-\beta} P_- \calZ_\omega \calY_{\varphi}  \calZ_\omega^* P_+  q^{-\beta}
\\
=
\bigl(q^{-\beta}(\calZ_\omega P_-   - P_- \calZ_\omega)q^{-\beta}\bigr) 
\calY_{\varphi}  (q^\beta P_+ q^{-\beta}) \calZ_\omega^*
\\
+
(q^{-\beta} P_- q^\beta)\calZ_\omega \calY_{\varphi} 
\bigl(q^{-\beta}(P_+ \calZ_\omega^*- \calZ_\omega^*P_+)q^{-\beta}\bigr)
\end{multline*}
is compact in $L^2 (\bbR; \calN)$. Here we have taken into account that
the operators $q^{-\beta}$ commute with $\calZ_\omega$ and $\calY_{\varphi}$.
 Recall that, by Proposition~\ref{Muck}, the operators $q^{-\beta} P_\pm q^\beta$ are bounded. 
 Therefore it suffices to use the fact that, by
  Lemma~\ref{lma.d1} (see also  Remark~\ref{rem.e2c}) applied to the operator valued function $\Xi (\lambda) = Z(\lambda)\omega(\lambda)$,
the operators 
$$
q^{-\beta}(\calZ_\omega P_\pm-P_\pm\calZ_\omega)q^{-\beta}:
L^2(\bbR;\calH)\to L^2(\bbR; \calN)
$$
are compact.  \qed

\subsection{Proof of Lemma~\ref{lma.d4k}}\label{sec.d2dw}

Putting together formulas \eqref{e9}, \eqref{eq:symb} and \eqref{eq:xi}, we see that it suffices to check that the operator valued function
\begin{equation}
\Psi (\lambda)=
Z(\lambda)Y(\lambda+i0) Z^* (\lambda)  \varphi(\lambda)\omega^2(\lambda)  -\sum_{\ell=1}^L \varkappa_\ell (S(\lambda_\ell) -I)\zeta(\lambda-\lambda_\ell)
\label{eq:symb2}\end{equation}
satisfies the assumptions of Lemma~\ref{lma.d1}. Clearly, both terms here are continuous functions of $\lambda$ away from  the set $\lambda_{1},\ldots, \lambda_{L}$ and tend to zero as $|\lambda|\to \infty$. 

Observe that the functions $\varphi$ and $\theta$ have the same jump \eqref{b8b}
at the point $\lambda_{\ell} $. Therefore the jump  of the first term in \eqref{eq:symb2} 
at the point $\lambda_{\ell}$ equals
\begin{equation}
\varkappa_\ell Z(\lambda_{\ell})Y(\lambda_{_\ell}+i0) Z^*(\lambda_{\ell})  .
\label{eq:symb2r}\end{equation}
Since $\zeta (\pm 0) = \pm 1/2$, the jump at the point $\lambda_{\ell} $ of the sum in \eqref{eq:symb2} equals
\begin{equation}
\varkappa_\ell (S(\lambda_\ell) -I).
\label{eq:symb2t}\end{equation}
 It follows   from   the representation \eqref{b8} of the scattering matrix $S(\lambda)$  that
expressions  \eqref{eq:symb2r} and \eqref{eq:symb2t} coincide.
 Hence the operator valued function $\Psi (\lambda)$ is continuous  at each point $\lambda_{\ell}$.

Finally, $\Psi (\lambda)$ satisfies condition \eqref{d0}  because $Z(\lambda)$ and $Y(\lambda+i0)$ are H\"older continuous functions, $\theta(\lambda)$ satisfies condition \eqref{b8a} and $\zeta(\lambda)$ satisfies condition \eqref{c6x}.
\qed

\appendix
\section{Proof of Lemma~\ref{lma.f5}}

\begin{proof}
1. First we recall some background information on the Legendre function. 
This function can be defined     
(see formulas (2.10.2) and (2.10.5) in  \cite{BE})  in terms of the hypergeometric function
 \begin{equation}
F(a,b,c;z)
= 
\sum_{n=0}^\infty \frac{(a)_n (b)_n}{ (c)_n}
\frac{z^n}{n!}, \quad 
(a)_n=\frac{\Gamma(a+n)}{\Gamma(a)}, \quad |z| <1,
\label{eq:hyper}
\end{equation}
where $\Gamma$ is the gamma-function. Namely, for $x>1$, we have
 \begin{equation}
P_{-\frac12+i\tau}(x)
=
\Re \Big(m(\tau)
F(\tfrac14-i\tfrac{\tau}{2},\tfrac34-i\tfrac{\tau}{2};1-i\tau; x^{-2})
x^{-\frac12 + i\tau}\Big)
\label{f16}
\end{equation}
  where
\begin{equation}
m(\tau)
=
\frac{\Gamma( i\tau)}{\sqrt{\pi}\Gamma(\tfrac12+ i\tau)}
2^{\tfrac12+i\tau}. 
\label{eq:dd}\end{equation}

Putting together \eqref{eq:hyper} and \eqref{f16}, we see that
\begin{equation}
P_{-\frac12+i\tau}(x)= \Re \big(m(\tau) x^{-\frac12 + i\tau}
\sum_{n=0}^\infty p_{n}(\tau) x^{-2n} \big),
\label{eq:Le}\end{equation}
where $p_{0}(\tau)=1$ and
$$
p_{n}(\tau)=\frac{(\tfrac14-i\tfrac{\tau}{2})_n (\tfrac34-i\tfrac{\tau}{2})_n}{ (1-i\tau)_n n!}.
$$
According to the Stirling formula all coefficients $p_n(\tau)$ 
are uniformly (in $n$ and $\tau$) bounded 
for $\tau$ in compact intervals $\delta\subset{\bbR}_{+}$. Moreover,
$$
| \partial p_{n}(\tau)/ \partial\tau| \leq C \ln n 
$$
where $C$ does not depend on $n$ and $\tau\in \delta$.
In particular, we see that $P_{-\frac12+i\tau}(x)$ is a smooth function of $x > 1$ and it has the asymptotics
  \begin{equation}
P_{-\frac12+i\tau}(x)
=
\Re \big( m(\tau)x^{-\frac12+ i\tau}\big)
+ O(x^{-5/2}), \quad x\to\infty;
\label{eq:f12}
\end{equation}
this asymptotics can be differentiated in $x$. 
The series \eqref{eq:Le} and hence the asymptotics \eqref{eq:f12} can also be differentiated in $\tau$.

Instead of \eqref{f16}, in a neighborhood of the point $x=1$ we use another representation (see formula (3.2.2)  in  \cite{BE})
$$
P_{-\frac12+i\tau}(x)=F(\tfrac12 - i\tau , \tfrac12+i\tau ;1;\tfrac{1-x}{2}) .
$$
It implies that   
\begin{equation}
|P_{-\frac12+i\tau}(x)|+ |P_{-\frac12+i\tau}'(x)|\leq C,\quad x\in [1,2] .
\label{eq:le1}\end{equation}

2.
Let us return to the function $w_{\tau}(\nu)$. Let us write \eqref{eq:Xx1} as
$$
 \sqrt{2\pi}   e^{-i \lambda}w_{\tau}(\lambda) 
=
\int_1^\infty  
P_{-\frac12+i\tau}(x)e^{-i \lambda x} dx.
$$
Integrating here by parts, we see that
$$
i \sqrt{2\pi}e^{-i \lambda} \lambda w_{\tau}(\lambda)
=
  P_{-1/2+ i\tau} (1)e^{-i \lambda}
+   \int_{1}^\infty e^{-i \lambda x}P_{-1/2+ i\tau}' (x) dx.
$$
According to \eqref{eq:f12} $|P_{-1/2+ i\tau}' (x)|\leq C x ^{-3/2}$ if $x\geq 2$.
According to \eqref{eq:le1}  the function $ P_{-1/2+ i\tau}' (x)$ is  bounded if $x\in [1,2]$. 
Therefore the integral in the right-hand side is   
bounded uniformly in $\lambda$ which yields the first estimate \eqref{f4}.  

3. 
To obtain the first estimate \eqref{f5}, we observe that, again  by \eqref{eq:le1}, the function $ P_{-1/2+ i\tau} (x)$ is   bounded for $x\in [1,2]$. For   $x\geq 2$ we use asymptotics \eqref{eq:f12}. Note that the leading term
 \begin{equation}
\int_{2}^\infty e^{-i \lambda x} x^{-1/2+i\tau}dx
= \abs{\lambda}^{-1/2-i\tau}
\int_{2\abs{\lambda}}^\infty e^{\mp i   y} y^{-1/2+i\tau}dy,\quad \pm \lambda>0,
\label{eq:hy2}
\end{equation}
satisfies estimate \eqref{f5}. The contribution of
the remainder $O(x^{-5/2})$ in \eqref{eq:f12} to the integral in \eqref{eq:Xx1} 
is uniformly   bounded.

4.
Estimates \eqref{f4} and \eqref{f5} on the derivative $\partial w_{\tau}(\lambda)/ \partial \tau$ can be obtained quite similarly because asymptotics  \eqref{eq:f12} are differentiable  in $\tau$ and estimates  \eqref{eq:le1}  remain true for the derivative $ \partial P_{-1/2+ i\tau} (x)/ \partial \tau$.
  The only difference is that instead of \eqref{eq:hy2} we now have the integral 
$$
\int_{2}^\infty e^{-i \lambda x} x^{-1/2+i\tau}\ln x\,  dx
$$
which is bounded by $\abs{\lambda}^{-1/2} |\ln \abs{\lambda} |$.
 \end{proof}

 \section*{Acknowledgements} 
Our collaboration has become possible through the hospitality and financial support 
of the Departments of Mathematics of the University of Rennes 1 and of KingÕs College London. 
The second author was partially supported by the projects NONa (ANR-08-BLANC-0228) and 
NOSEVOL (ANR-11-BS0101901).
The authors are grateful for hospitality and financial
support to the  Mittag-Leffler Institute, Sweden, where the paper was completed 
during the authors' stay in autumn 2012.


\begin{thebibliography}{00}

\bibitem{BS}
{\sc M.~S{\rm h}.~Birman and M.~Z.~Solomyak,} 
\emph{Double operator integrals in a Hilbert space},
Integral Eq. Oper. Theory, \textbf{47} (2003), 131--168. 

\bibitem{BE}
{\sc A.~Erd\'elyi,  W.~Magnus, F.~Oberhettinger, F.~G.~Tricomi,}
\emph{Higher transcendental functions.  } Vol. 1,  
McGraw-Hill, New York-Toronto-London, 1953.


\bibitem{Howland1}
{\sc J.~Howland,}
\emph{Spectral theory of selfadjoint Hankel matrices,}
Michigan Math. J. \textbf{33} (1986), no. 2, 145--153.

\bibitem{Howland2}
{\sc J.~Howland,}
\emph{Spectral theory of operators of Hankel type. I, II,} 
Indiana Univ. Math. J. \textbf{41} (1992), no. 2, 409--426, 427--434.


\bibitem{Kato2}
{\sc T.~Kato,} 
\emph{Wave operators and similarity for some non-selfadjoint operators,}
Math. Ann. \textbf{162} (1966), 258--279.

\bibitem{KM}
{\sc V.~Kostrykin, K.~Makarov}, 
\emph{On Krein's example,}
Proc. Amer. Math. Soc. \textbf{136} (2008), no. 6, 2067--2071.

\bibitem{Kuroda}
{\sc S.~T.~Kuroda}, 
\emph{Scattering theory for differential operators,}
J. Math. Soc. Japan \textbf{25} (1973)
\emph{I Operator theory,} 75--104,
\emph{II Self-adjoint elliptic operators,}
222--234.

\bibitem{Muck}
{\sc B.~Muckenhoupt}, \emph{Weighted norm inequalities for the Hardy maximal function,} Trans. AMS \textbf{165}   (1972), 207--226.


\bibitem{Peller1}
{\sc V.~V.~Peller,} 
\emph{Hankel operators in perturbation theory of unitary and self-adjoint operators}, 
Funct. Anal. Appl. {\bf 19}  (1985), 111--123.

\bibitem{Peller}
{\sc V.~Peller,} 
\emph{Hankel operators and their applications.}
Springer-Verlag, New York, 2003.


\bibitem{Power}
{\sc S.~Power,}
\emph{Hankel operators with discontinuous symbol,} 
Proc. Amer. Math. Soc. \textbf{65} (1977), no. 1, 77--79.


\bibitem{Push3}
{\sc A.~Pushnitski,}
\emph{The spectral shift function and the invariance principle,}
J. Functional Analysis, \textbf{183}, no.~2 (2001), 269--320. 

\bibitem{Push}
{\sc A.~Pushnitski,}
\emph{The scattering matrix and the differences of spectral projections},
Bull. London Math. Soc. \textbf{40}  (2008), 227--238.

\bibitem{Push2}
{\sc A.~Pushnitski,} 
\emph{Spectral theory of discontinuous functions of self-adjoint operators: essential spectrum},
Integral Eq. Oper. Theory, \textbf{68} (2010), 75--99. 

\bibitem{Push1}
{\sc A.~Pushnitski,} \emph{Scattering matrix and functions of self-adjoint operators}, 
Journal of Spectral Theory \textbf{1} (2011), 221--236.

\bibitem{PY1}
{\sc A.~Pushnitski, D.~Yafaev,}
\emph{Spectral theory of 
discontinuous functions of self-adjoint operators  
and scattering theory},
J. Funct. Anal. \textbf{259} (2010), 1950--1973.


\bibitem{PY2}
{\sc A.~Pushnitski, D.~Yafaev,}
\emph{A   multichannel scheme in smooth scattering theory},
preprint, arXiv:1209.3238.


\bibitem {RS} {\sc M.~Reed and B.~Simon}, {\em Methods of modern mathematical physics. IV, Analysis of operators}.
Academic Press, San Diego, CA,  1979.


\bibitem{Yafaev}
{\sc D.~R.~Yafaev,}
\emph{Mathematical scattering theory. General theory.}
Amer. Math. Soc., Providence, RI, 1992. 


\bibitem{Yafaev3}
{\sc D.~R.~Yafaev,}
\emph{A commutator method for the diagonalization 
of Hankel operators,} Funct. Anal. Appl. \textbf{44} (2010), no. 4, 295--306.




\end{thebibliography}
\end{document}